\documentclass{article}
\usepackage{macros}

\title{The smooth locus in infinite-level Rapoport-Zink spaces} 
\author{Alexander Ivanov and Jared Weinstein}

\begin{document}
\maketitle

\begin{abstract}  Rapoport-Zink spaces are deformation spaces for $p$-divisible groups with additional structure.  At infinite level, they become preperfectoid spaces.  Let $\cM_{\infty}$ be an infinite-level Rapoport-Zink space of EL type, and let $\cM_{\infty}^\circ$ be one connected component of its geometric fiber.  We show that $\cM_{\infty}^{\circ}$ contains a dense open subset which is cohomologically smooth in the sense of Scholze.  This is the locus of $p$-divisible groups which do not have any extra endomorphisms.  As a corollary, we find that the cohomologically smooth locus in the infinite-level modular curve $X(p^\infty)^{\circ}$ is exactly the locus of elliptic curves $E$ with supersingular reduction, such that the formal group of $E$ has no extra endomorphisms.
\end{abstract}

\section{Main theorem}
Let $p$ be a prime number.  Rapoport-Zink spaces \cite{RapoportZink} are deformation spaces of $p$-divisible groups 
equipped with some extra structure.  This article concerns the geometry of Rapoport-Zink spaces of EL type (endomorphisms + level structure).  In particular we consider the infinite-level spaces $\cM_{\caD,\infty}$, which are preperfectoid spaces \cite{ScholzeWeinstein}.  An example is the space $\cM_{H,\infty}$, where $H/\overline{\bF}_p$ is a $p$-divisible group of height $n$.  The points of $\cM_{H,\infty}$ over a nonarchimedean field $K$ containing $W(\overline{\bF}_p)$ are in correspondence with isogeny classes of $p$-divisible groups $G/\OO_K$ equipped with a quasi-isogeny $G\otimes_{\OO_K} \OO_K/p \to H\otimes_{\overline{\bF}_p} \OO_K/p$ and an isomorphism $\Q_p^n\isom VG$ (where $VG$ is the rational Tate module).  

The infinite-level space $\cM_{\caD,\infty}$ appears as the limit of finite-level spaces, each of which is a smooth rigid-analytic space.  We would like to investigate the question of smoothness for the space $\cM_{\caD,\infty}$ itself, which is quite a different matter.  We need the notion of cohomological smoothness \cite{ScholzeEtaleCohomology}, which makes sense for general morphisms of analytic adic spaces, and which is reviewed in Section \ref{CohomologicalSmoothness}.  Roughly speaking, an adic space is cohomologically smooth over $C$ (where $C/\Q_p$ is complete and algebraically closed) if it satisfies local Verdier duality.  In particular, if $U$ is a quasi-compact adic space which is cohomologically smooth over $\Spa (C,\OO_C)$, then the cohomology group $H^i(U,\bF_\ell)$ is finite for all $i$ and all primes $\ell\neq p$.

Our main theorem shows that each connected component of the geometric fiber of $\cM_{\caD,\infty}$ has a dense open subset which is cohomologically smooth.

\begin{thm} \label{MainTheorem}
Let $\caD$ be a basic EL datum (cf. Section \ref{RZReview}).  Let $C$ be a complete algebraically closed extension of the field of scalars of $\cM_{\caD,\infty}$, and let $\cM_{\caD,\infty}^\circ$ be a connected component of the base change $\cM_{\caD,\infty,C}$.  Let $\cM_{\caD,\infty}^{\circ,\ns}\subset \cM_{\caD,\infty}^{\circ}$ be the non-special locus (cf. Section \ref{SpecialLocus}), corresponding to $p$-divisible groups without extra endomorphisms.  Then $\cM_{\caD,\infty}^{\circ,\ns}$ is cohomologically smooth over $C$.
\end{thm}

We remark that outside of trivial cases, $\pi_0(\cM_{\caD,\infty,C})$ has no isolated points, which implies that no open subset of $\cM_{\caD,\infty,C}$ can be cohomologically smooth.  (Indeed, the $H^0$ of any quasi-compact open fails to be finitely generated.)  Therefore it really is necessary to work with individual connected components of the geometric fiber of $\cM_{\caD,\infty}$.  

Theorem \ref{MainTheorem} is an application of the perfectoid version of the Jacobian criterion for smoothness, due to Fargues--Scholze \cite{FarguesScholze};  cf. Theorem \ref{TheoremCohomologicalSmoothnessCriterion}.  The latter theorem involves the Fargues-Fontaine curve $X_C$ (reviewed in Section \ref{SectionFFCurve}).  It asserts that a functor $\cM$ on perfectoid spaces over $\Spa(C,\OO_C)$ is cohomologically smooth, when $\cM$ can be interpreted as global sections of a smooth morphism $Z\to X_C$, subject to a certain condition on the tangent bundle $\Tan_{Z/X_C}$. 

In our application to Rapoport-Zink spaces, we construct a smooth morphism $Z\to X_C$, whose moduli space of global sections is isomorphic to $\cM_{\caD,\infty}^\circ$ (Lemma \ref{lm:MZ_and_MHinfty}).  Next, we show that a geometric point $x \in \cM_{\caD,\infty}^\circ(C)$ lies in $\cM_{\caD,\infty}^{\circ,{\rm non-sp}}(C)$ if and only if the corresponding section $s \colon X_C \rightarrow Z$ satisfies the condition that all slopes of the vector bundle $s^\ast\Tan_{Z/X_C}$ on $X_C$ are positive (Theorem \ref{TheoremHNSlopes}).   This is exactly the condition on $\Tan_{Z/X_C}$ required by Theorem \ref{TheoremCohomologicalSmoothnessCriterion}, so we can conclude that $\cM_{\caD,\infty}^\circ$ is cohomologically smooth.


The geometry of Rapoport-Zink spaces is related to the geometry of Shimura varieties.  As an example, consider the tower of classical modular curves $X(p^\infty)$, considered as rigid spaces over $C$.  There is a perfectoid space $X(p^\infty)$ over $C$ for which $X(p^\infty)\sim \varprojlim_n X(p^n)$, and a Hodge-Tate period map $\pi_{HT}\from X(p^\infty)\to \bP^1_C$ \cite{ScholzeTorsion}, which is $\GL_2(\Q_p)$-equivariant.  
Let $X(p^\infty)^\circ\subset X(p^\infty)$ be a connected component.

\begin{cor} \label{CorollaryModularCurve} The following are equivalent for a $C$-point $x$ of $X(p^\infty)^\circ$.  
\begin{enumerate}
\item The point $x$ corresponds to an elliptic curve $E$, such that the $p$-divisible group $E[p^\infty]$ has $\End E[p^\infty]=\Z_p$.  
\item The stabilizer of $\pi_{HT}(x)$ in $\PGL_2(\Q_p)$ is trivial.
\item There is a neighborhood of $x$ in $X(p^\infty)^\circ$ which is cohomologically smooth over $C$.
\end{enumerate}
\end{cor}

\section{Review of Rapoport-Zink spaces at infinite level}
\label{RZReview}

\subsection{The infinite-level Rapoport-Zink space $\cM_{H,\infty}$}

Let $k$ be a perfect field of characteristic $p$, and let $H$ be a $p$-divisible group of heignt $n$ and dimension $d$ over $k$.  We review here the definition of the infinite-level Rapoport-Zink space associated with $H$.  

First there is the formal scheme $\cM_H$ over $\Spf W(k)$ parametrizing deformations of $H$ up to isogeny, as in \cite{RapoportZink}.  For a $W(k)$-algebra $R$ in which $p$ is nilpotent, $\cM_H(R)$ is the set of isomorphism classes of pairs $(G,\rho)$, where $G/R$ is a $p$-divisible group and $\rho\from H\otimes_k R/p \to G\otimes_R R/p$ is a quasi-isogeny.  

The formal scheme $\cM_H$ locally admits a finitely generated ideal of definition.  Therefore it makes sense to pass to its adic space $\cM_H^{\ad}$, which has generic fiber $(\cM_H^{\ad})_\eta$, a rigid-analytic space over $\Spa(W(k)[1/p],W(k))$.   Then $(\cM_H^{\ad})_{\eta}$ has the following moduli interpretation:  it is the sheafification of the functor assigning to a complete affinoid $(W(k)[1/p],W(k))$-algebra {$(R,R^+)$ the set of pairs $(G,\rho)$, where $G$ is a $p$-divisible group defined over an open and bounded subring $R_0\subset R^+$, and $\rho\from H\otimes_k R_0/p\to G\otimes_{R_0} R_0/p$ is a quasi-isogeny.  There is an action of $\Aut H$ on $\cM_H^{\ad}$ obtained by composition with $\rho$.

Given such a pair $(G,\rho)$, Grothendieck-Messing theory produces a surjection $M(H)\otimes_{W(k)} R \to \Lie G[1/p]$ of locally free $R$-modules, where $M(H)$ is the covariant Dieudonn\'e module.  There is a Grothendieck-Messing period map $\pi_{GM}\from (\cM_H^{\ad})_{\eta}\to \caF\ell$, where $\caF\ell$ is the rigid-analytic space  parametrizing rank $d$ locally free quotients of $M(H)[1/p]$.   The  morphism $\pi_{GM}$ is equivariant for the action of $\Aut H$.  It has open image $\caF\ell^a$ (the admissible locus).

We obtain a tower of rigid-analytic spaces over $(\cM_H^{\ad})_{\eta}$ by adding level structures.  For a complete affinoid $(W(k)[1/p],W(k))$-algebra $(R,R^+)$, and an element of $(\cM_H^{\ad})_{\eta}(R,R^+)$ represented locally on $\Spa(R,R^+)$ by a pair $(G,\rho)$ as above, we have the Tate module $TG=\varprojlim_m G[p^m]$, considered as an adic space over $\Spa(R,R^+)$ with the structure of a $\Z_p$-module \cite[(3.3)]{ScholzeWeinstein}.  Finite-level spaces $\cM_{H,m}$ are obtained by trivializing the $G[p^m]$;  these are finite \'etale covers of $(\cM_H^{\ad})_{\eta}$.   The infinite-level space is obtained by trivializing all of $TG$ at once, as in the following definition. 

\begin{Def}[{\cite[Definition 6.3.3]{ScholzeWeinstein}}] Let $\cM_{H,\infty}$ be the functor which sends a complete affinoid $(W(k)[1/p],W(k))$-algebra $(R,R^+)$ to the set of triples $(G,\rho,\alpha)$, where $(G,\rho)$ is an element of $(\cM_H)^{\ad}_{\eta}(R,R^+)$, and $\alpha\from \Z_p^n \to TG$ is a $\Z_p$-linear map which is an isomorphism pointwise on $\Spa(R,R^+)$.  
\end{Def}

There is an equivalent definition in terms of {\em isogeny} classes of triples $(G,\rho,\alpha)$, where this time $\alpha\from \Q_p^n\to VG$ is a trivialization of the rational Tate module.  Using this definition, it becomes clear that $\cM_{H,\infty}$ admits an action of the product $\GL_n(\Q_p)\times \Aut^0 H$, where $\Aut^0$ means automorphisms in the isogeny category.  Then the period map $\pi_{GM}\from \cM_{H,\infty}\to \caF\ell$ is equivariant for $\GL_n(\Q_p)\times \Aut^0 H$, where $\GL_n(\Q_p)$ acts trivially on $\caF\ell$.  

We remark that $\cM_{H,\infty}\sim \varprojlim_m \cM_{H,m}$ in the sense of \cite[Definition 2.4.1]{ScholzeWeinstein}.

One of the main theorems of \cite{ScholzeWeinstein} is the following.

\begin{thm} \label{TheoremPreperfectoid} The adic space $\cM_{H,\infty}$ is a preperfectoid space.  
\end{thm}

This means that for any perfectoid field $K$ containing $W(k)$, the base change $\cM_{H,\infty}\times_{\Spa(W(k)[1/p],W(k))} \Spa(K,\OO_K)$ becomes perfectoid after $p$-adically completing.

We sketch here the proof of Theorem \ref{TheoremPreperfectoid}.  Consider the ``universal cover'' $\tilde{H}=\varprojlim_p H$ as a sheaf of $\Q_p$-vector spaces on the category of $k$-algebras.  This has a canonical lift to the category of $W(k)$-algebras \cite[Proposition 3.1.3(ii)]{ScholzeWeinstein}, which we continue to call $\tilde{H}$.  The adic generic fiber $\tilde{H}^{\ad}_{\eta}$ is a preperfectoid space, as can be checked ``by hand'':  it is a product of the $d$-dimensional preperfectoid open ball $(\Spa W(k)\powerseries{T_1^{1/p^\infty},\dots,T_d^{1/p^\infty}})_{\eta}$ by the constant adic space $VH^{\et}$, where $H^{\et}$ is the \'etale part of $H$.  
Given a triple $(G,\rho,\alpha)$ representing an element of $\cM_{H,\infty}(R,R^+)$, the quasi-isogeny $\rho$ induces an isomorphism $\tilde{H}^{\ad}_{\eta}\times_{\Spa(W(k)[1/p],W(k))} \Spa(R,R^+)\to \tilde{G}^{\ad}_{\eta}$;  composing this with $\alpha$ gives a morphism $\Q_p^n\to \tilde{H}^{\ad}_{\eta}(R,R^+)$.  We have therefore described a morphism $\cM_{H,\infty} \to (\tilde{H}^{\ad}_{\eta})^n$.  

Theorem \ref{TheoremPreperfectoid} follows from the fact that the morphism $\cM_{H,\infty}\to (\tilde{H}^{\ad})_{\eta}^n$ presents $\cM_{H,\infty}$ as an open subset of a Zariski closed subset of $(\tilde{H}^{\ad})_{\eta}^n$.  We conclude this subsection by spelling out how this is done.  We have a {\em quasi-logarithm} map $\qlog_H\from \tilde{H}^{\ad}_{\eta} \to M(H)[1/p] \otimes_{W(k)[1/p]} \bG_a$ \cite[Definition 3.2.3]{ScholzeWeinstein}, a $\Q_p$-linear morphism of adic spaces over $\Spa(W(k)[1/p],W(k))$.

Now suppose $(G,\rho)$ is a deformation of $H$ to $(R,R^+)$.  The logarithm map on $G$ fits into an exact sequence of $\Z_p$-modules:
\[
0 \to G_\eta^{\ad}[p^{\infty}](R,R^+) \to G_{\eta}^{\ad}(R,R^+) \to \Lie G[1/p].\]
After taking projective limits along multiplication-by-$p$, this turns into an exact sequence of $\Q_p$-vector spaces,
\[
0 \to VG(R,R^+) \to \tilde{G}_{\eta}^{\ad}(R,R^+) \to \Lie G[1/p].\]
On the other hand, we have a commutative diagram
\[
\xymatrix{
\tilde{H}_{\eta}(R,R^+) \ar[r]^{\isom} \ar[d]_{\qlog_H} & \tilde{G}_{\eta}(R,R^+) \ar[d]^{\log_G} \\
M(H)\otimes_{W(k)} R \ar[r] & \Lie G[1/p].
}
\]
The lower horizontal map $M(H)\otimes_{W(k)} R\to \Lie G[1/p]$ is the quotient by the $R$-submodule of $M(H)\otimes_{W(k)} R$ generated by the image of $VG(R,R^+) \to \tilde{G}_{\eta}^{\ad}(R,R^+)\isom \tilde{H}_{\eta}^{\ad}(R,R^+)\to M(H)\otimes_{W(k)} R$.   

Thus if we have a triple $(G,\rho,\alpha)$ representing an element of $\cM_{H,\infty}(R,R^+)$, then we have a map $\Q_p^n\to \tilde{H}_{\eta}^{\ad}(R,R^+)$, such that the cokernel of $\Q_p^n\to \tilde{H}_{\eta}^{\ad}(R,R^+)\to M(H)\otimes_{W(k)} R$ is a projective $R$-module of rank $d$, namely $\Lie G[1/p]$.  This condition on the cokernel allows us to formulate an alternate description of $\cM_{H,\infty}$ which is independent of deformations.

\begin{prop} \label{PropositionAlternateDescriptionOfM} The adic space $\cM_{H,\infty}$ is isomorphic to the functor which assigns to a complete affinoid $(W(k)[1/p],W(k))$-algebra $(R,R^+)$ the set of $n$-tuples $(s_1,\dots,s_n)\in \tilde{H}^{\ad}_{\eta}(R,R^+)^n$ such that the following conditions are satisfied:
\begin{enumerate}
\item The quotient of $M(H)\otimes_{W(k)} R$ by the span of the $\qlog(s_i)$ is a projective $R$-module $W$ of rank $d$.
\item For all geometric points $\Spa(C,\OO_C)\to \Spa(R,R^+)$, the sequence 
\[ 0 \to \Q_p^n\stackrel{(s_1,\dots,s_n)}{\to} \tilde{H}^{\ad}_{\eta}(C,\OO_C)\to W\otimes_R C \to 0\]
is exact.
\end{enumerate}
\end{prop}

\subsection{Infinite-level Rapoport-Zink spaces of EL type}\label{sec:infiniteRZEL}
This article treats the more general class of Rapoport-Zink spaces of EL type.  We review these here.

\begin{Def} 
\label{DefinitionRationalELDatum}
Let $k$ be an algebraically closed field of characteristic $p$.  A {\em{rational EL datum}} is a quadruple $\caD=(B,V,
H,\mu)$, where
\begin{itemize}
\item $B$ is a semisimple $\Q_p$-algebra,
\item $V$ is a finite $B$-module,
\item $H$ is an object of the isogeny category of $p$-divisible groups over $k$, equipped with an action $B\to \End H$,
\item $\mu$ is a conjugacy class of $\overline{\Q}_p$-rational cocharacters $\mathbf{G}_m\to \bG$, where $\bG/\Q_p$ is the algebraic group $\GL_B(V)$.
\end{itemize}
These are subject to the conditions:
\begin{itemize} 
\item If $M(H)$ is the (rational) Dieudonn\'e module of $H$, then there exists an isomorphism $M(H)\cong V\otimes_{\Q_p} W(k)[1/p]$ of $B\otimes_{\Q_p} W(k)[1/p]$-modules.  In particular $\dim V=\height H$.
\item In the weight decomposition of $V\otimes_{\Q_p} \overline{\Q}_p\cong \bigoplus_{i\in \Z} V_i$ determined by $\mu$, only weights 0 and 1 appear, and $\dim V_0=\dim H$.
\end{itemize}
\end{Def}

The {\em reflex field} $E$ of $\caD$ is the field of definition of the conjugacy class $\mu$.  We remark that the weight filtration (but not necessarily the weight decomposition) of $V\otimes_{\Q_p}\overline{\Q}_p$ may be descended to $E$, and so we will be viewing $V_0$ and $V_1$ as $B\otimes_{\Q_p} E$-modules.  

The infinite-level Rapoport-Zink space $\cM_{\caD,\infty}$ is defined in \cite{ScholzeWeinstein} in terms of moduli of deformations of the $p$-divisible group $H$ along with its $B$-action.  It admits an alternate description along the lines of Proposition \ref{PropositionAlternateDescriptionOfM}.

\begin{prop}[{\cite[Theorem 6.5.4]{ScholzeWeinstein}}] Let $\caD=(B,V,H,\mu)$ be a rational EL datum.  Let $\breve{E}=E\cdot W(k)$.  Then $\cM_{\caD,\infty}$ is isomorphic to the functor which inputs a complete affinoid $(\breve{E},\OO_{\breve{E}})$-algebra $(R,R^+)$ and outputs the set of $B$-linear maps 
\[ s\from V\to \tilde{H}_\eta^{\ad}(R,R^+), \]
subject to the following conditions.
\begin{itemize}
\item Let $W$ be the quotient 
\[ V\otimes_{\Q_p} R\stackrel{\qlog_H\circ s}{\longrightarrow} M(H)\otimes_{W(k)} R \to W \to 0.\]
Then $W$ is a finite projective $R$-module, which locally on $R$ is isomorphic to $V_0\otimes_{E} R$ as a $B\otimes_{\Q_p} R$-module.
\item For any geometric point $x=\Spa(C,\OO_C)\to\Spa(R,R^+)$, the sequence of $B$-modules
\[0\to V \to \tilde{H}(\OO_C) \to W\otimes_R C \to 0\]
is exact.
\end{itemize}
\end{prop}

If $\caD = (\Q_p, \Q_p^n, H, \mu)$, where $H$ has height $n$ and dimension $d$ and $\mu(t)=(t^{\oplus d},1^{\oplus (n-d)})$, then $E=\Q_p$ and $\cM_{\caD,\infty} = \cM_{H,\infty}$.

In general, we call $\breve E$ the field of scalars of $\cM_{\caD,\infty}$, and for a complete algebraically closed extension $C$ of $\breve E$, we write $\cM_{\caD,\infty,C} = \cM_{\caD,\infty} \times_{\Spa(\breve E,\OO_{\breve E})} \Spa(C,\OO_C)$ for the corresponding geometric fiber of $\cM_{\caD,\infty}$.

The space $\cM_{\caD,\infty}$ admits an action by the product group $\bG(\Q_p)\times J(\Q_p)$, where $J/\Q_p$ is the algebraic group $\Aut_B^\circ(H)$.  A pair $(\alpha,\alpha')\in \bG(\Q_p)\times J(\Q_p)$ sends $s$ to $\alpha'\circ s\circ \alpha^{-1}$.  

There is once again a Grothendieck-Messing period map $\pi_{GM}\from \cM_{\caD,\infty}\to \caF\ell_\mu$ onto the rigid-analytic variety whose $(R,R^+)$-points parametrize $B\otimes_{\Q_p} R$-module quotients of $M(H)\otimes_{W(k)} R$ which are projective over $R$, and which are of type $\mu$ in the sense that they are (locally on $R$) isomorphic to $V_0\otimes_E R$.  The morphism $\pi_{GM}$ sends an $(R,R^+)$-point of $\cM_{\caD,\infty}$ to the quotient $W$ of $M(H)\otimes_{W(k)} R$ as above.  It is equivariant for the action of $\bG(\Q_p)\times J(\Q_p)$, where $\bG(\Q_p)$ acts trivially on $\caF\ell_\mu$.  In terms of deformations of the $p$-divisible group $H$, the period map $\pi_{GM}$ sends a deformation $G$ to $\Lie G$. 

There is also a Hodge-Tate period map $\pi_{HT}\from \cM_{\caD,\infty}\to \caF\ell_{\mu}'$, where $\caF\ell_\mu'(R,R^+)$ parametrizes $B\otimes_{\Q_p} R$-module quotients of $V\otimes_{\Q_p} R$ which are projective over $R$, and which are (locally on $R$) isomorphic to $V_1\otimes_E R$.   The morphism $\pi_{HT}$ sends an $(R,R^+)$-point of $\cM_{\caD,\infty}$ to the image of $V\otimes_{\Q_p} R\to M(H)\otimes_{W(k)} R$.  It is equivariant for the action of $\bG(\Q_p)\times J(\Q_p)$, where this time $J(\Q_p)$ acts trivially on $\caF\ell_\mu'(R,R^+)$.  In terms of deformations of the $p$-divisible group $H$, the period map $\pi_{HT}$ sends a deformation $G$ to $(\Lie G^\vee)^\vee$.

\section{The Fargues-Fontaine curve}
\label{SectionFFCurve}
\subsection{Review of the curve}  We briefly review here some constructions and results from \cite{FarguesFontaineCurve}.  First we review the absolute curve, and then we cover the version of the curve which works in families.

Fix a perfectoid field $F$ of characteristic $p$, with $F^\circ\subset F$ its ring of integral elements.  Let $\varpi\in F^\circ$ be a pseudo-uniformizer for $F$, and let $k$ be the residue field of $F$.   Let $W(F^\circ)$ be the ring of Witt vectors, which we equip with the $(p,[\varpi])$-adic topology.  Let $\caY_F=\Spa(W(F^\circ),W(F^\circ))\backslash\set{\abs{p[\varpi]}=0}$.  Then $\caY_F$ is an analytic adic space over $\Q_p$.  The Frobenius automorphism of $F$ induces an automorphism $\phi$ of $\caY_F$.  Let $B_F=H^0(\caY_F,\OO_{\caY_F})$, a $\Q_p$-algebra endowed with an action of $\phi$.  Let $P_F$ be the graded ring $P_F=\bigoplus_{n\geq 0} B_F^{\phi=p^n}$.  Finally, the Fargues-Fontaine curve is $X_F=\Proj P_F$.  It is shown in \cite{FarguesFontaineCurve} that $X_F$ is the union of spectra of Dedekind rings, which justifies the use of the word ``curve'' to describe $X_F$.  Note however that there is no ``structure morphism'' $X_F\to \Spec F$.  

If $x\in X_F$ is a closed point, then the residue field of $x$ is a perfectoid field $F_x$ containing $\Q_p$ which comes equipped with an inclusion $i\from F\injects F_x^\flat$, which presents $F_x^\flat$ as a finite extension of $F$. Such a pair $(F_x,i)$ is called an untilt of $F$.  Then $x\mapsto (F_x,i)$ is a bijection between closed points of $X_F$ and isomorphism classes of untilts of $F$, modulo the action of Frobenius on $i$.  Thus if $F=E^\flat$ is the tilt of a given perfectoid field $E/\Q_p$, then $X_{E^\flat}$ has a canonical closed point $\infty$, corresponding to the untilt $E$ of $E^\flat$.

An important result in \cite{FarguesFontaineCurve} is the classification of vector bundles on $X_F$.  (By a vector bundle on $X_F$ we are referring to a locally free $\OO_{X_F}$-module $\caE$ of finite rank.  We will use the notation $V(\caE)$ to mean the corresponding geometric vector bundle over $X_F$, whose sections correspond to sections of $\caE$.)  Recall that an {\em isocrystal} over $k$ is a finite-dimensional vector space $N$ over $W(k)[1/p]$ together with a Frobenius semi-linear automorphism $\phi$ of $N$.   Given $N$, we have the graded $P_F$-module $\bigoplus_{n\geq 0} (N\otimes_{W(k)[1/p]} B_F)^{\phi=p^n}$, which corresponds to a vector bundle $\caE_F(N)$ on $X_F$.  Then the Harder-Narasimhan slopes of $\caE_F(N)$ are negative to those of $N$.  If $F$ is algebraically closed, then every vector bundle on $X_F$ is isomorphic to $\caE_F(N)$ for some $N$.  

It is straightforward to ``relativize'' the above constructions.  If $S=\Spa(R,R^+)$ is an affinoid perfectoid space over $k$, one can construct the adic space $\caY_S$, the ring $B_S$, the scheme $X_S$, and the vector bundles $\caE_S(N)$ as above.  Frobenius-equivalences classes of untilts of $S$ correspond to effective Cartier divisors of $X_S$ of degree 1.  

In our applications, we will start with an affinoid perfectoid space $S$ over $\Q_p$.  We will write $X_S=X_{S^\flat}$, and we will use $\infty$ to refer to the canonical Cartier divisor of $X_{S}$ corresponding to the untilt $S$ of $S^\flat$.  Thus if $N$ is an isocrystal over $k$, and $S=\Spa(R,R^+)$ is an affinoid perfectoid space over $W(k)[1/p]$, then the fiber of $\caE_{S}(N)$ over $\infty$ is $N\otimes_{W(k)[1/p]} R$.  

Let $S = \Spec(R,R^+)$ be as above and let $\infty$ be the corresponding Cartier divisor. We denote the completion of the ring of functions on $\mathcal{Y}_S$ along $\infty$ by $B_{\dR}^+(R)$. It comes equipped with a surjective homomorphism $\theta \colon B_{\dR}^+(R) \rightarrow R$, whose kernel is a principal ideal $\ker(\theta) = (\xi)$. 

\subsection{Relation to $p$-divisible groups}

Here we recall the relationships between $p$-divisible groups and global sections of vector bundles on the Fargues-Fontaine curve.  Let us fix a perfect field $k$ of characteristic $p$, and write $\Perf_{W(k)[1/p]}$ for the category of perfectoid spaces over $W(k)[1/p]$. Given a $p$-divisible group $H$ over $k$ with covariant isocrystal $N$, if $H$ has slopes $s_1,\dots,s_k \in \mathbb{Q}$, then $N$ has the slopes $1-s_1, \dots,1-s_k$. For an object $S$ in $\Perf_{W(k)[1/p]}$ we define the vector bundle $\caE_S(H)$ on $X_S$ by 
\[ \caE_S(H) = \caE_S(N) \otimes_{\OO_{X_S}} \OO_{X_S}(1). \]
Under this normalization, the Harder-Narasimhan slopes of $\caE_S(H)$ are (pointwise on $S$) the same as the slopes of $H$.  

Let us write $H^0(\caE(H))$ for the sheafification of the functor on $\Perf_{W(k)[1/p]}$, which sends $S$ to $H^0(X_S,\caE_S(H))$. 

\begin{prop} \label{BanachColmezSpace} Let $H$ be a $p$-divisible group over a perfect field $k$ of characteristic $p$, with isocrystal $N$.  There is an isomorphism $\tilde{H}^{\ad}_{\eta}\isom H^0(\caE(H))$ of sheaves on $\Perf_{W(k)[1/p]}$ 
making the diagram commute:
\[\xymatrix{
\tilde{H}^{\ad}_{\eta} \ar[rr] \ar[dr]_{\qlog_H} && 
H^0(\caE(H)) \ar[dl] \\
& N\otimes_{W(k)[1/p]} \bG_a,
}
\]
where the morphism $H^0(\caE(H))\to N\otimes_{W(k)[1/p]} \bG_a$ sends a global section of $\caE(H)$ to its fiber at $\infty$.  
\end{prop}

\begin{proof}  Let $S=\Spa(R,R^+)$ be an affinoid perfectoid space over $W(k)[1/p]$.  Then $\tilde{H}^{\ad}_{\eta}(R,R^+)\isom\tilde{H}(R^\circ)\isom\tilde{H}(R^\circ/p)$.  Observe that $\tilde{H}(R^\circ/p)=\Hom_{R^\circ/p}(\Q_p/\Z_p,H)[1/p]$, where the Hom is taken in the category of $p$-divisible groups over $R^\circ/p$.  Recall the crystalline Dieudonn\'e functor $G\mapsto M(G)$ from $p$-divisible groups to Dieudonn\'e crystals \cite{Messing}.  Since the base ring $R^{\circ}/p$ is semiperfect, the latter category is equivalent to the category of finite projective modules over Fontaine's period ring $A_{\rm{cris}}(R^{\circ}/p)=A_{\rm{cris}}(R^\circ)$, equipped with Frobenius and Verschiebung. 

Now we apply \cite[Theorem A]{ScholzeWeinstein}:  since $R^{\circ}/p$ is f-semiperfect, the crystalline Dieudonn\'e functor is fully faithful up to isogeny.  Thus
\[ \Hom_{R^\circ/p}(\Q_p/\Z_p,H)[1/p] \isom \Hom_{A_{\rm{cris}}(R^\circ),\phi}(M(\Q_p/\Z_p),M(H))[1/p], \]
where the latter Hom is in the category of modules over $A_{\rm{cris}}(R^\circ)$ equipped with Frobenius.  Recall that $B_{\rm{cris}}^+(R^\circ)=A_{\rm{cris}}(R^\circ)[1/p]$. Since $H$ arises via base change from $k$, we have $M(H)[1/p]=B_{\rm{cris}}^+(R^\circ) \otimes_{W(k)[1/p]} N $.  For its part, $M(\Q_p/\Z_p)[1/p]=B_{\rm{cris}}^+(R^\circ)e$, for a basis element $e$ on which Frobenius acts as $p$.  Therefore
\[ \tilde{H}(R^{\circ})\isom ( B_{\rm{cris}}^+(R^\circ) \otimes_{W(k)[1/p]} N )^{\phi=p}.\]

On the Fargues-Fontaine curve side, we have by definition $H^0(X_S,\caE_S(H))=(B_S\otimes_{W(k)[1/p]} N)^{\phi =p}$.  The isomorphism between $(B_S\otimes_{W(k)[1/p]} N)^{\phi =p}$ and $(B_{\rm{cris}}^+(R^\circ) \otimes_{W(k)[1/p]} N )^{\phi=p}$ is discussed in \cite[Remarque 6.6]{LeBras}.   

The commutativity of the diagram in the proposition is \cite[Proposition 5.1.6(ii)]{ScholzeWeinstein}, at least in the case that $S$ is a geometric point, but this suffices to prove the general case.  
\end{proof}

With Proposition \ref{BanachColmezSpace} we can reinterpret the infinite-level Rapoport Zink spaces as moduli spaces of {\em modifications} of vector bundles on the Fargues-Fontaine curve.  First we do this for $\cM_{H,\infty}$.   In the following, we consider $\cM_{H,\infty}$ as a sheaf on the category of perfectoid spaces over $W(k)[1/p]$.  

\begin{prop}
\label{DescriptionOfMH}
 Let $H$ be a $p$-divisible group of height $n$ and dimension $d$ over a perfect field $k$.  Let $N$ be the associated isocrystal over $k$.  Then $\cM_{H,\infty}$ is isomorphic to the functor which inputs an affinoid perfectoid space $S=\Spa(R,R^+)$ over $W(k)[1/p]$ and outputs the set of exact sequences
\begin{equation}
\label{EquationXSExactSequence}
 0 \to \OO_{X_S}^n \stackrel{s}{\to} \caE_S(H) \to i_{\infty *} W \to 0,
 \end{equation}
where $i_\infty\from \Spec R\to X_S$ is the inclusion, and $W$ is a projective $\OO_S$-module quotient of $N\otimes_{W(k)[1/p]} \OO_S$ of rank $d$.
\end{prop}  

\begin{proof}
We briefly describe this isomorphism on the level of points over $S=\Spa(R,R^+)$.  Suppose that we are given a point of $\cM_{H,\infty}(S)$, corresponding to a $p$-divisible group $G$ over $R^\circ$, together with a quasi-isogeny $\iota\from H\otimes_{k} R^\circ/p \to G\otimes_{R^\circ} R^\circ/p$ and an isomorphism $\alpha\from \Q_p^n\to VG$ of sheaves of $\Q_p$-vector spaces on $S$.  The logarithm map on $G$ fits into an exact sequence of sheaves of $\Z_p$-modules on $S$,
\[
0 \to G_\eta^{\ad}[p^\infty] \to G_{\eta}^{\ad} \to \Lie G[1/p] \to 0.\]
After taking projective limits along multiplication-by-$p$, this turns into an exact sequence of sheaves of $\Q_p$-vector spaces on $S$,
\[
0 \to VG \to \tilde{G}_{\eta}^{\ad} \to \Lie G[1/p] \to 0.\]

The quasi-isogeny induces an isomorphism $\tilde{H}_{\eta}^{\ad} \times_{\Spa W(k)[1/p]} S \isom \tilde{G}_{\eta}^{\ad}$;  composing this with the level structure gives an injective map $\Q_p^n\to \tilde{H}^{\ad}_{\eta}(S)$, whose cokernel $W$ is isomorphic to the projective $R$-module $\Lie G$ of rank $d$.  In light of Theorem \ref{BanachColmezSpace}, the map $\Q_p^n\to \tilde{H}^{\ad}_{\eta}(S)$ corresponds to an $\OO_{X_{S}}$-linear map $s\from \OO_{X_S}^n\to \caE_S(H)$, which fits into the exact sequence in \eqref{EquationXSExactSequence}.
\end{proof}

Similarly, we have a description of $\cM_{\caD,\infty}$ in terms of modifications.

\begin{prop}
\label{PropDescriptionOfMD}
Let $\caD=(B,V,H,\mu)$ be a rational EL datum.
Then $\cM_{\caD,\infty}$ is isomorphic to the functor which inputs an affinoid perfectoid space $S$ over $\breve{E}$ and outputs the set of exact sequences of  $B \otimes_{\Q_p} \OO_{X_S}$-modules
\[ 0 \to V\otimes_{\Q_p} \OO_{X_S} \stackrel{s}{\to} \caE_S(H) \to i_{\infty*} W \to 0,\]
where $W$ is a finite projective $\OO_S$-module, which is locally isomorphic to $V_0\otimes_{\Q_p} \OO_S$ as a $B\otimes_{\Q_p} \OO_S$-module (using notation from Definition \ref{DefinitionRationalELDatum}). 
\end{prop}

\subsection{The determinant morphism, and connected components}
If we are given a rational EL datum $\caD$, there is a {\em determinant morphism} $\det \from \cM_{\caD,\infty}\to \cM_{\det \caD,\infty}$, which we review below.  For an algebraically closed perfectoid field $C$ containing $W(k)[1/p]$, the base change $\cM_{\det \caD,\infty,C}$ is a locally profinite set of copies of $\Spa C$.  For a point $\tau\in \cM_{\det \caD,\infty}(C)$, let $\cM_{\caD,\infty}^{\tau}$ be the fiber of $\cM_{\caD,\infty}\to \cM_{\det\caD,\infty}$ over $\tau$.  We will prove in Section \ref{ProofOfMainTheorem} that each $\cM_{\caD,\infty}^{\tau,\ns}$ is cohomologically smooth if $\caD$ is basic. This implies that $\pi_0(\cM_{\caD,\infty}^{\tau,\ns})$ is discrete, so that cohomogical smoothness of $\cM_{\caD,\infty}^{\tau,\ns}$ is inherited by each of its connected components.  This is Theorem \ref{MainTheorem}.  In certain cases (for example Lubin-Tate space) it is known that $\cM_{\caD,\infty}^{\tau}$ is already connected \cite{ChenConnectedComponents}.

We first review the determinant morphism for the space $\cM_{H,\infty}$, where $H$ is a $p$-divisible group of height $n$ and dimension $d$ over a perfect field $k$ of characteristic $p$. Let $\breve{E}=W(k)[1/p]$.   For a perfectoid space $S$ over $\breve{E}$, we have the vector bundle $\caE_S(H)$ and its determinant $\det \caE_S(H)$, a line bundle of degree $d$.  (This does not correspond to a $p$-divisible group ``$\det H$'' unless $d\leq 1$.)  We define $\cM_{\det H,\infty}(S)$ to be the functor which inputs a perfectoid space $S=\Spa(R,R^+)$ over $\breve{E}$ and outputs the set of morphisms $s\from \OO_{X_S}\to \det \caE_S(H)$, such that the cokernel of $s$ is a projective $B_{\dR}^+(R)/(\xi)^d$-module of rank 1, where $(\xi)$ is the kernel of $B_{\dR}^+(R)\to R$.    Then for an algebraically closed perfectoid field $C/\breve{E}$, the set $\cM_{\det H,\infty}(C)$ is a $\Q_p^\times$-torsor.  The morphism $\det\from \cM_{H,\infty}\to \cM_{\det H,\infty}$ is simply $s\mapsto \det s$.

For the general case, let $\caD=(B,V,H,\mu)$ be a rational EL datum.  Let $F=Z(B)$ be the center of $B$.  Then $F$ is a semisimple commutative $\Q_p$-algebra, and $V$ is free as an $F$-module.   We put $\bG=\Aut_B(V)$ (as an algebraic group), and then $\bG^{\text{ab}}=\bG/\bG^{\text{der}}=\Aut_F(\det_F V)\isom \Res_{F/\Q_p}\bG_m$.  Let $\mu^{\text{ab}}$ be the composition of $\mu$ with $\bG\to \bG^{\text{ab}}$.  Let $\cM_{\det \caD,\infty}$ be the functor which inputs a perfectoid space $S=\Spa(R,R^+)$ over $\breve{E}$ and outputs the set of $F$-linear morphisms $s\from \det_F V\otimes_{\Q_p} \OO_{X_S}\to \det_F \caE_S(H)$, such that the cokernel of $s$ is of the form $i_{\infty\ast} W$, where $W$ is a finite projective $\OO_S$-module, locally isomorphic to $\det_F V_0 \otimes_{\Q_p} F$ as a $F \otimes_{\Q_p} \OO_S$-module.

\subsection{Basic Rapoport-Zink spaces}
\label{Basic}
The main theorem of this article concerns basic Rapoport-Zink spaces, so we recall some facts about these here.  

Let $H$ be a $p$-divisible group over a perfect field $k$ of characteristic $p$.  The space $\caM_{H,\infty}$ is said to be basic when the $p$-divisible group $H$ (or rather, its Dieudonn\'e module $M(H)$) is isoclinic.  This is equivalent to saying that the natural map 
\[ \End^{\circ} H \otimes_{\Q_p} W(k)[1/p] \to \End_{W(k)[1/p]} M(H)[1/p] \]
is an isomorphism, where on the right the endomorphisms are not required to commute with Frobenius.


More generally we have a notion of basicness for a rational EL datum $(B,H,V,\mu)$, referring to the following equivalent conditions:
\begin{itemize} 
\item The $\bG$-isocrystal ($\bG=\Aut_B V$) associated to $H$ is basic in the sense of Kottwitz \cite{KottwitzIsocrystals}.  
\item The natural map 
\[ \End^\circ_B(H) \otimes_{\Q_p} W(k)[1/p] \to \End_{B\otimes_{\Q_p} W(k)[1/p]} M(H)[1/p] \]
is an isomorphism.
\item Considered as an algebraic group over $\Q_p$, the automorphism group $J=\Aut^\circ_B H$ is an inner form of $\bG$.  
\item Let $D'=\End^\circ_B H$.  For any algebraically closed perfectoid field $C$ containing $W(k)$, the map 
\[ D'\otimes_{\Q_p} \OO_{X_C} \to \SEnd_{(B\otimes_{\Q_p} \OO_{X_C})} \caE_C(H) \]
is an isomorphism.
\end{itemize}

In brief, the duality theorem from \cite{ScholzeWeinstein} says the following.  Given a basic EL datum $\caD$, there is a dual datum $\check{\caD}$, for which the roles of the groups $\bG$ and $J$ are reversed.  There is a $\bG(\Q_p)\times J(\Q_p)$-equivariant isomorphism $\cM_{\caD,\infty}\isom \cM_{\check{\caD},\infty}$ which exchanges the roles of $\pi_{GM}$ and $\pi_{HT}$.

\subsection{The special locus}
\label{SpecialLocus}
Let $\caD=(B,V,H,\mu)$ be a basic rational EL datum relative to a perfect field $k$ of characteristic $p$, with reflex field $E$.  Let $F$ be the center of $B$.  
Define $F$-algebras $D$ and $D'$ by 
\begin{eqnarray*}
D&=& \End_{B} V\\
D'&=& \End_{B} H
\end{eqnarray*}
Finally, let $\bG=\Aut_B V$ and $J=\Aut_B H$, considered as  algebraic groups over $\Q_p$.  Then $\bG$ and $J$ both contain $\Res_{F/\Q_p}\bG_m$.  

Let $C$ be an algebraically closed perfectoid field containing $\breve{E}$, and let $x\in \cM_{\caD,\infty}(C)$.  Then $x$ corresponds to a $p$-divisible group $G$ over $\OO_C$ with endomorphisms by $B$, and also it corresponds to a $B\otimes_{\Q_p}\OO_{X_C}$-linear map $s\from V\otimes_{\Q_p}\OO_X\to \caE_C(N)$ as in Proposition \ref{PropDescriptionOfMD}.  Define $A_x=\End_B G$ (endomorphisms in the isogeny category).  Then $A_x$ is a semisimple $F$-algebra. In light of Proposition \ref{PropDescriptionOfMD}, an element of $A_x$ is a pair $(\alpha,\alpha')$, where $\alpha\in \End_{B\otimes_{\Q_p}\OO_{X_C}} V\otimes \OO_{X_C}=\End_B V=D$ and $\alpha'\in \End_{B\otimes_{\Q_p} \OO_{X_C}} \cE_C(H)=D'$ (the last equality is due to basicness), such that $s\circ\alpha=\alpha'\circ s$.  Thus:
  \[ A_x\isom \set{(\alpha,\alpha')\in D\times D' \biggm\vert  s\circ\alpha = \alpha'\circ s}. \]

\begin{lm} The following are equivalent:
\begin{enumerate}
\item The $F$-algebra $A_x$ strictly contains $F$.
\item The stabilizer of $\pi_{GM}(x)\in \mathcal{F}\ell_\mu(C)$ in $J(\Q_p)$ strictly contains $F^\times$.
\item The stabilizer of $\pi_{HT}(x)\in \mathcal{F}\ell_\mu'(C)$ in $\bG(\Q_p)$ strictly contains $F^\times$.  
\end{enumerate}
\end{lm}
\begin{proof}
As in Proposition \ref{PropDescriptionOfMD}, let $s \colon V\otimes_{\Q_p} \OO_{X_S} \stackrel{s}{\to} \caE_S(H)$ be the modification corresponding to $x$. 

Note that the condition (1) is equivalent to the existence of an invertible element $(\alpha,\alpha') \in A_x$ not contained in (the diagonally embedded) $F$.  Also note that if one of $\alpha,\alpha'$ lies in $F$, then so does the other, in which case they are equal.   

Suppose $(\alpha,\alpha')\in A_x$ is invertible.   The point $\pi_{GM}(x)\in \caF\ell_\mu$ corresponds to the cokernel of the fiber of $s$ at $\infty$.  Since $\alpha'\circ s=s\circ \alpha$, the cokernels of $\alpha'\circ s$ and $s$ are the same, which means exactly that $\alpha'\in J(\Q_p)$ stabilizes $\pi_{GM}(x)$.  Thus (1) implies (2).  Conversely, if there exists $\alpha'\in J(\Q_p)\backslash F^\times$ which stabilizes $\pi_{GM}(x)$, it means that the $B\otimes_{\Q_p}\OO_{X_C}$-linear maps $s$ and $\alpha'\circ s$ have the same cokernel, and therefore there exists $\alpha\in \End_{B\otimes_{\Q_p}\OO_{X_C}} V\otimes_{\Q_p} \OO_{X_C}=D$ such that $s\circ \alpha = \alpha'\circ s$, and then $(\alpha,\alpha')\in A_x\backslash F^\times$.  This shows that (2) implies (1).  

The equivalence between (1) and (3) is proved similarly. 
\end{proof}

\begin{Def} The {\em special locus} in $\cM_{\caD,\infty}$ is the subset $\cM_{\caD,\infty}^{\spe}$ defined by the condition $A_x\neq F$.  The {\em non-special locus} $\cM_{\caD,\infty}^{\ns}$ is the complement of the special locus.
\end{Def}

The special locus is built out of ``smaller'' Rapoport-Zink spaces, in the following sense.  Let $A$ be a semisimple $F$-algebra, equipped with two $F$-embeddings $A\to D$ and $A\to D'$, so that $A\otimes_F B$ acts on $V$ and $H$.  Also assume that a cocharacter in the conjugacy class $\mu$ factors through a cocharacter $\mu_0\from \Gm\to \Aut_{A\otimes_F B} V$.   Let $\caD_0=(A\otimes_F B,V,H,\mu_0)$.  Then there is an evident morphism $\cM_{\caD_0,\infty}\to \cM_{\caD,\infty}$.  The special locus $\cM_{\caD,\infty}^{\spe}$ is the union of the images of all the $\cM_{\caD_0,\infty}$, as $A$ ranges through all semisimple $F$-subalgebras of $D\times D'$ strictly containing $F$.

\section{Cohomological smoothness}
\label{CohomologicalSmoothness}

Let $\Perf$ be the category of perfectoid spaces in characteristic $p$, with its pro-\'etale topology \cite[Definition 8.1]{ScholzeEtaleCohomology}.  For a prime $\ell\neq p$, there is a notion of $\ell$-cohomological smoothness \cite[Definition 23.8]{ScholzeEtaleCohomology}.  We only need the notion for morphisms $f\from Y'\to Y$ between sheaves on $\Perf$ which are separated and representable in locally spatial diamonds.  If such an $f$ is $\ell$-cohomologically smooth, and $\Lambda$ is an $\ell$-power torsion ring, then the relative dualizing complex $Rf^!\Lambda$ is an invertible object in $D_{\et}(Y',\Lambda)$ (thus, it is v-locally isomorphic to $\Lambda[n]$ for some $n\in \bZ$), and the natural transformation $Rf^!\Lambda\otimes f^*\to Rf^!$ of functors $D_{\et}(Y,\Lambda)\to D_{\et}(Y',\Lambda)$ is an equivalence \cite[Proposition 23.12]{ScholzeEtaleCohomology}.  In particular, if $f$ is projection onto a point, and $Rf^!\Lambda\isom \Lambda[n]$, one derives a statement of Poincar\'e duality for $Y'$:
\[ R\Hom(R\Gamma_c(Y',\Lambda),\Lambda)\isom R\Gamma(Y',\Lambda)[n].\]

We will say that $f$ is cohomologically smooth if it is $\ell$-cohomologically smooth for all $\ell\neq p$.  
As an example, if $f\from Y'\to Y$ is a separated smooth morphism of rigid-analytic spaces over $\Q_p$, then the associated morphism of diamonds $f^\diamond\from (Y')^{\diamond}\to Y^{\diamond}$ is cohomologically smooth \cite[Proposition 24.3]{ScholzeEtaleCohomology}.  
There are other examples where $f$ does not arise from a finite-type map of adic spaces.
For instance, if $\tilde{B}_C=\Spa C\left< T^{1/p^\infty} \right>$ is the perfectoid closed ball over an algebraically closed perfectoid field $C$, then $\tilde{B}_C$ is cohomologically smooth over $C$.

If $Y$ is a perfectoid space over an algebraically closed perfectoid field $C$, it seems quite difficult to detect whether $Y$ is cohomologically smooth over $C$.  We will review in Section \ref{JacobianCriterionPerfectoid} a ``Jacobian criterion'' from \cite{FarguesScholze} which applies to certain kinds of $Y$.  But first we give a classical analogue of this criterion in the context of schemes.

\subsection{The Jacobian criterion:  classical setting}

\begin{prop} \label{PropositionJacobianCriterion} Let $X$ be a smooth projective curve over an algebraically closed field $k$.  Let $Z\to X$ be a smooth morphism.   Define $\cM_Z$ to be the functor which inputs a $k$-scheme $T$ and outputs the set of sections of $Z\to X$ over $X_T$, that is, the set of morphisms $s$ making
\[
\xymatrix{
  &  Z \ar[d] \\
X \times_k T \ar[ur]^s \ar[r]  & X 
}
\]
commute, subject to the condition that, fiberwise on $T$, the vector bundle $s^*\Tan_{Z/X}$ has vanishing $H^1$. Then $\cM_Z \to \Spec k$ is formally smooth.
\end{prop}

Here $\Tan_{Z/X}$ is the tangent bundle, equal to the $\OO_Z$-linear dual of the sheaf of differentials $\Omega_{Z/X}$, which is locally free of finite rank.  Let $\pi\from X\times_k T \to T$ be the projection.  For $t\in T$, let $X_t$ be the fiber of $\pi$ over $t$, and let $s_t\from X_t\to Z$ be the restriction of $s$ to $X_t$.   By proper base change, the fiber of $R^1\pi_*s^*\Tan_{Z/X}$ at $t\in T$ is $H^1(X_t,s_t^*\Tan_{Z/X})$.  The condition about the vanishing of $H^1$ in the proposition is equivalent to $H^1(X_t,s_t^*\Tan_{Z/X})=0$ for each $t\in T$.  By Nakayama's lemma, this condition is equivalent to $R^1\pi_*s^*\Tan_{Z/X}=0$.  

\begin{proof}  Suppose we are given a commutative diagram
\begin{equation}
\label{EquationMZDiagram}
\xymatrix{
T_0 \ar[r] \ar[d] & \cM_Z \ar[d] \\
T \ar[r] & \Spec k,
}
\end{equation}
where $T_0\to T$ is a first-order thickening of affine schemes; thus $T_0$ is the vanishing locus of a square-zero ideal sheaf $I\subset \OO_T$.  Note that $I$ becomes an $\OO_{T_0}$-module.

The morphism $T_0\to \cM_Z$ in \eqref{EquationMZDiagram} corresponds to a section of $Z\to X$ over $T_0$.  Thus there is a solid diagram
\begin{equation}
\label{EquationZXdiagram}
\xymatrix{
X\times_k T_0 \ar[r]^{s_{0}} \ar[d] & Z \ar[d] \\
X\times_k T     \ar[r] \ar@{.>}[ur]^{s} \ar[r] & X.
}
\end{equation}
We claim that there exists a dotted arrow making the diagram commute.  
Since $Z\to X$ is smooth, it is formally smooth, and therefore this arrow exists Zariski-locally on $X$.  Let $\pi\from X\times_k T\to T$ and $\pi_0\from X\times_k T_0\to T_0$ be the projections.  Then $X\times_k T_0$ is the vanishing locus of the ideal sheaf $\pi^*I\subset \OO_{X\times_k T}$.  Note that sheaves of sets on $X\times_k T$ are equivalent to sheaves of sets on $X\times_k T_0$;  under this equivalence, $\pi^*I$ and $\pi_0^*I$ correspond.  By \cite[\href{http://stacks.math.columbia.edu/tag/04BU}{Remark 36.9.6}]{StacksProject}, the set of such morphisms form a (Zariski) sheaf of sets on $X\times_k T$, which when viewed as a sheaf on $X\times_k T_0$ is a torsor for 
\[\SHom_{\OO_{X\times_k T_0}}(s_{0}^*\Omega_{Z/X}, \pi_0^*I)\cong s_0^*\Tan_{Z/X} \otimes \pi_0^*I.\]
This torsor corresponds to class in 
\[H^1(X\times_k T_0, s_0^*\Tan_{Z/X}\otimes \pi_0^*I).\]
This $H^1$ is the limit of a spectral sequence with terms
\[ H^p(T_0, R^q\pi_{0*}(s_0^*\Tan_{Z/X}\otimes\pi_0^*I)).\]
But since $T_0$ is affine and $R^q\pi_{0*}(s_0^*\Tan_{Z/X}\otimes\pi_0^*I)$ is quasi-coherent, the above terms vanish for all $p>0$, and therefore
\[H^1(X\times_k T_0, s_0^*\Tan_{Z/X}\otimes \pi_0^*I)
\isom H^0(T_0,R^1\pi_{0*}(s_0^*\Tan_{Z/X}\otimes \pi_0^*I)).\]
Since $s_0^*\Tan_{Z/X}$ is locally free, we have $s_0^*\Tan_{Z/X}\otimes\pi_0^*I \isom s_0^*\Tan_{Z/X}\otimes^{\bL} \pi_{0*}I$, and we may apply the projection formula \cite[\href{http://stacks.math.columbia.edu/tag/08ET}{Lemma 35.21.1}]{StacksProject} to obtain 
\[ R\pi_{0*}(s_0^*\Tan_{Z/X}\otimes \pi_0^*I)\isom R\pi_{0*}s_0^*\Tan_{Z/X} \otimes^{\bL} I. \]
Now we apply the hypothesis about vanishing of $H^1$, which implies that $R\pi_{0*}s_0^*\Tan_{Z/X}$ is quasi-isomorphic to the locally free sheaf $\pi_{0*}s_0^*\Tan_{Z/X}$ in degree 0.  Therefore the complex displayed above has $H^1=0$.  

Thus our torsor is trivial, and so a morphism $s\from X\times_k T\to Z$ exists filling in \eqref{EquationZXdiagram}.  The final thing to check is that $s$ corresponds to a morphism $T\to \cM_Z$, i.e., that it satisfies the fiberwise $H^1=0$ condition.  But this is automatic, since $T_0$ and $T$ have the same schematic points.  
\end{proof}

In the setup of Proposition \ref{PropositionJacobianCriterion}, let $s\from X \times_k \cM_Z \to Z$ be the universal section.    That is, the pullback of $s$ along a morphism $T \to \cM_Z$ is the section $X\times_k T\to Z$ to which this morphism corresponds.   Let $\pi\from X\times_k \cM_Z\to \cM_Z$ be the projection. By Proposition \ref{PropositionJacobianCriterion} $\cM_Z\to \Spec k$ is formally smooth. There is an isomorphism 
\[ \pi_*s^*\Tan_{Z/X} \isom \Tan_{\cM_Z/\Spec k}. \]
Indeed, the proof of Proposition \ref{PropositionJacobianCriterion} shows that $\pi_*s^*\Tan_{Z/X}$ has the same universal property with respect to first order deformations as $\Tan_{\cM_Z/\Spec k}$.

The following example is of similar spirit as our main application of the perfectoid Jacobian criterion below. 

\begin{ex}
Let $X = \bP^1$ over the algebraically closed field $k$. For $d \in \bZ$, let $V_d = \underline{\smash{\mathrm{Spec}}}_X{\rm \Sym}_{\caO_X}(\caO(-d))$ be the geometric vector bundle over $X$ whose global sections are $\Gamma(X,\caO(d))$. Fix integers $n,d,\delta > 0$ and let $P$ be a homogeneous polynomial over $k$ of degree $\delta$ in $n$ variables. Then $P$ defines a morphism $P \colon \prod_{i=1}^n V_d \rar V_{d\delta}$, by sending sections $(s_i)_{i=1}^n$ of $V_d$ to the section $P(s_1,\dots,s_n)$ of $V_{d\delta}$. Fix a global section $f \colon X \rar V_{d\delta}$ to the projection morphism and consider the pull-back of $P$ along $f$:
\[
\xymatrix{
Z \ar@{^(->}[r] & P^{-1}(f) \ar[r] \ar[d] & X \ar[d]^f \ar[rd]^{\id_X}\\
& \prod_{i=1}^n V_d \ar[r]^P  & V_{d\delta} \ar[r] & X 
}
\]
Moreover, let $Z$ be the smooth locus of $P^{-1}(f)$ over $X$. It is an open subset. The derivatives $\frac{\partial P}{\partial x_i}$ of $P$ are homogeneous polynomials of degree $\delta - 1$ in $n$ variables, hence can be regarded as functions $\prod_{i=1}^n V_d \rightarrow V_{d(\delta - 1)}$. A point $y \in P^{-1}(f)$ lies in $Z$ if and only if $\frac{\partial P}{\partial x_i}(y)$, $i=1,\dots,n$ are not all zero.
We wish to apply Proposition \ref{PropositionJacobianCriterion} to $Z/X$. Let $\cM'_Z$ denote the space of global sections of $Z$ over $X$, that is for a $k$-scheme $T$, $\cM_Z'(T)$ is the set of morphisms $s \colon X \times _k T \rar Z$ as in the proposition (without any further conditions). A $k$-point $g \in \cM'_Z(k)$ corresponds to a section $g \colon X \rar \prod_{i=1}^n  V_d$, satisfying $P \circ g = f$. In general, for a (geometric) vector bundle $V$ on $X$ with corresponding locally free $\caO_X$-module $\cE$, the pull-back of  the tangent space $\mathrm{Tan}_{V/X}$ along a section $s \colon X \rar V$ is canonically isomorphic to $\cE$. Hence in our situation (using that $Z \subseteq P^{-1}(f)$ is open) the tangent space $g^{\ast} \mathrm{Tan}_{Z/X}$ can be computed from the short exact sequence,
\[
0 \rar g^{\ast} \mathrm{Tan}_{Z/X} \rar \bigoplus_{i=1}^n \caO(d) \stackrel{D_g P}{\longrar} \caO(d\delta) \rar 0,
\]
where $D_g P$ is the derivative of $P$ at $g$. It is the $\caO_X$-linear map given by $(t_i)_{i=1}^n \mapsto \sum_{i = 1}^n \frac{\partial P}{\partial x_i}(g)t_i$ (note that $\frac{\partial P}{\partial x_i}(g)$ are global sections of $\caO(d(\delta - 1))$). Note that $D_g P$ is surjective: by Nakayama, it suffices to check this fiberwise, where it is true by the condition defining $Z$.

The space $\cM_Z$ is the subfunctor of $\cM'_Z$ consisting of all $g$ such that (fiberwise) $g^{\ast}\mathrm{Tan}_{Z/X} = \mathrm{ker}(D_g P)$ has vanishing $H^1$. Writing $\mathrm{ker}(D_g P) = \bigoplus_{i=1}^r \caO(m_i)$ ($m_i \in \bZ$), this is equivalent to $m_i \geq -1$. By the Proposition \ref{PropositionJacobianCriterion} we conclude that $\cM_Z$ is formally smooth over $k$.

Consider now a numerical example. Let $n = 3$, $d = 1$ and $\delta = 4$ and let $g \in \cM'_Z(k)$. Then $D_g P \in \mathrm{Hom}_{\caO_X}(\caO(1)^{\oplus 3}, \caO(4)) = \Gamma(X,\caO(3)^{\oplus 3})$, a $12$-dimensional $k$-vector space, and moreover, $D_g P$ lies in the open subspace of surjective maps. We have the short exact sequence of $\caO_X$-modules
\begin{equation}\label{eq:derivative_numerical_example}
0 \rar g^{\ast} \mathrm{Tan}_{Z/X} \rar \caO(1)^{\oplus 3} \stackrel{D_g  P}{\longrar} \caO(4) \rar 0 
\end{equation}
This shows that $g^{\ast} \mathrm{Tan}_{Z/X}$ has rank $2$ and degree $-1$. Moreover, being a subbundle of $\caO(1)^{\oplus 3}$ it only can have slopes $\leq 1$. There are only two options, either $g^{\ast} \mathrm{Tan}_{Z/X} \cong \caO(-1) \oplus \caO$ or $g^{\ast} \mathrm{Tan}_{Z/X} \cong \caO(-2) \oplus \caO(1)$. The point $g$ lies in $\cM_Z$ if and only if the first option occurs for $g$. Which option occurs can be seen from the long exact cohomology sequence associated to \eqref{eq:derivative_numerical_example}:
\[0 \rar \Gamma(X, g^{\ast} \mathrm{Tan}_{Z/X}) \rar \underbrace{\Gamma(X, \caO(1))^{\oplus 3}}_{\text{$6$-dim'l}} \stackrel{\Gamma(D_g P)}{\longrar} \underbrace{\Gamma(X, \caO(4))}_{\text{$5$-dim'l}} \rar \coh^1(X, g^{\ast} \mathrm{Tan}_{Z/X}) \rar 0, 
\]
It is clear that $\Gamma(X, g^{\ast} \mathrm{Tan}_{Z/X})$ is $1$-dimesional if and only if $g^{\ast} \mathrm{Tan}_{Z/X} \cong \caO(-1) \oplus \caO$ and $2$-dimensional otherwise. The first option is generic, i.e., $\cM_Z$ is an open subscheme of $\cM'_Z$. 
\end{ex}

\subsection{The Jacobian criterion:  perfectoid setting}

We present here the perfectoid version of Proposition \ref{PropositionJacobianCriterion}.

\label{JacobianCriterionPerfectoid}

\begin{thm}[Fargues-Scholze {\cite{FarguesScholze}}]
\label{TheoremCohomologicalSmoothnessCriterion}
Let $S=\Spa(R,R^+)$ be an affinoid perfectoid space in characteristic $p$.  Let $Z\to X_S$ be a smooth morphism of schemes.  Let $\cM_Z^{>0}$ be the functor which inputs a perfectoid space $T\to S$ and outputs the set of sections of $Z\to X_S$ over $T$, that is, the set of morphisms $s$ making 
\[
\xymatrix{
  &  Z \ar[d] \\
X_T \ar[ur]^s \ar[r]  & X_S 
}
\]
commute, subject to the condition that, fiberwise on $T$, all Harder-Narasimhan slopes of the vector bundle $s^*\Tan_{Z/X_S}$ are positive. Then $\cM_Z^{>0}\to S$ is a cohomologically smooth morphism of locally spatial diamonds.  
\end{thm}

\begin{ex} Let $S=\eta=\Spa(C,\OO_C)$, where $C$ is an algebraically closed perfectoid field of characteristic $0$, and let $Z=\bV(\caE_S(H))\to X_S$ be the geometric vector bundle attached to $\caE_S(H)$, where $H$ is a $p$-divisible group over the residue field of $C$.  Then $\cM_Z=H^0(\caE_S(H))$ is isomorphic to $\tilde{H}_\eta^{\ad}$ by Proposition \ref{BanachColmezSpace}.  Let $s\from X_{\cM_Z}\to Z$ be the universal morphism;  then $s^*\Tan_{Z/X_S}$ is the constant Banach-Colmez space associated to $H$ (i.e., the pull-back of $\caE_S(H)$ along $X_{\cM_Z} \rar X_S$).  This has vanishing $H^1$ if and only if $H$ has no \'etale part.  This is true if and only if $\cM_Z^{>0}$ is isomorphic to a perfectoid open ball.  The perfectoid open ball is cohomologically smooth, in accord with Theorem \ref{TheoremCohomologicalSmoothnessCriterion}.  In contrast, if the \'etale quotient $H^{\et}$ has height $d>0$, then $\pi_0(\tilde{H}_{\eta}^{\ad})\isom \Q_p^d$ implies that $\tilde{H}_{\eta}^{\ad}$ is not cohomologically smooth.
\end{ex}

In the setup of Theorem \ref{TheoremCohomologicalSmoothnessCriterion}, suppose that $x= \Spa(C,\OO_C)\to S$ is a geometric point, and that $x\to \cM_Z^{>0}$ is an $S$-morphism, corresponding to a section $s\from X_C\to Z$.  Then $s^*\Tan_{Z/X_S}$ is a vector bundle on $X_C$.  In light of the discussion in the previous section, we are tempted to interpret $H^0(X_C,s^*\Tan_{Z/X_S})$ as the ``tangent space of $\cM_Z^{>0}\to S$ at $x$''.  At points $x$ where $s^*\Tan_{Z/X_S}$ has only positive Harder-Narasimhan slopes, this tangent space is a perfectoid open ball.  

\section{Proof of the main theorem}
\label{ProofOfMainTheorem}

\subsection{Dilatations and modifications}
\label{SectionDilatations}
As preparation for the proof of Theorem \ref{MainTheorem}, we review the notion of a dilatation of a scheme at a locally closed subscheme \cite[\S 3.2]{NeronModels}.  

Throughout this subsection, we fix some data.  Let $X$ be a curve, meaning that $X$ is a scheme which is locally the spectrum of a Dedekind ring.  Let $\infty\in X$ be a closed point with residue field $C$.  Let $i_\infty\from \Spec C \to X$ be the embedding, and let $\xi\in \OO_{X,\infty}$ be a local uniformizer at $\infty$.

\begin{prop} Let $V\to X$ be a morphism of finite type, and let $Y\subset V_\infty$ be a locally closed subscheme of the fiber of $V$ at $\infty$.

There exists a morphism of $X$-schemes $V'\to V$ 
which is universal for the following property: $V'\to X$ is flat at $\infty$, and $V_\infty'\to V_\infty$ factors through $Y\subset V_\infty$.  
\end{prop}

The $X$-scheme $V'$ is the {\em dilatation} of $V$ at $Y$.  We review here its construction.

First suppose that $Y\subset V_\infty$ is closed.   Let $\cI\subset \OO_V$ be the ideal sheaf which cuts out $Y$.  Let $B\to V$ be the blow-up of $V$ along $Y$.   Then $\cI\cdot \OO_B$ is a locally principal ideal sheaf.  The dilatation $V'$ of $V$ at $Y$ is the open subscheme of $B$ obtained by imposing the condition that the ideal $(\cI\cdot\OO_B)_x\subset \OO_{B,x}$ is generated by $\xi$ at all $x\in B$ lying over $\infty$.  

We give here an explicit local description of the dilatation $V'$.  Let $\Spec A$ be an affine neighborhood of $\infty$, such that $\xi\in A$, and let $\Spec R\subset V$ be an open subset lying over $\Spec A$.  Let $I=(f_1,\dots,f_n)$ be the restriction of $\cI$ to $\Spec R$, so that $I$ cuts out $Y\cap \Spec A$.  Then the restriction of $V'\to V$ to $\Spec R$ is $\Spec R'$, where 
\[ R'=R\left[\frac{f_1}{\xi},\dots,\frac{f_n}{\xi}\right]/(\xi\text{-torsion}). \]

Now suppose $Y\subset V_\infty$ is only locally closed, so that $Y$ is open in its closure $\overline{Y}$. Then the dilatation of $V$ at $Y$ is the dilatation of $V\backslash (\overline{Y}\backslash Y)$ at $Y$.

Note that a dilatation $V'\to V$ is an isomorphism away from $\infty$, and that it is affine.

\begin{ex}
\label{ExampleDilatationModification}
Let 
\[ 0 \to \caE' \to \caE \to i_{\infty *} W \to 0\]
be an exact sequence of $\OO_X$-modules, where $\caE$ (and thus $\caE'$) is locally free, and $W$ is a $C$-vector space.  (This is an elementary modification of the vector bundle $\caE$.)  Let $K=\ker(\caE_\infty\to W)$.  

Let $\bV(\caE)\to X$ be the geometric vector bundle corresponding to $\caE$.  Similarly, we have $\bV(\caE')\to X$, and an $X$-morphism $\bV(\caE')\to \bV(\caE)$.  Let $\bV(K)\subset \bV(\caE)_\infty$ be the affine space associated to $K\subset \caE_\infty$.   We claim that $\bV(\caE')$ is isomorphic to the dilatation $\bV(\caE)'$ of $\bV(\caE)$ at $\bV(K)$.  Indeed, by the universal property of dilatations, there is a morphism $\bV(\caE')\to \bV(\caE)'$, which is an isomorphism away from $\infty$.   

To see that $\bV(\caE')\to \bV(\caE)'$ is an isomorphism, it suffices to work over $\OO_{X,\infty}$.  Over this base, we may give a basis $f_1,\dots,f_n$ of global sections of $\caE$, with $f_1,\dots,f_k$ lifting a basis for $K\subset \caE_\infty$.  Then the localization of $\bV(\caE)'\to \bV(\caE)$ at $\infty$ is isomorphic to 
\[ \Spec \OO_{X,\infty}\left[\frac{f_1}{\xi},\dots,\frac{f_k}{\xi},f_{k+1},\cdots,f_n\right] \to \Spec \OO_{X,\infty}[f_1,\dots,f_n]. \]
This agrees with the localization of $\bV(\caE')\to \bV(\caE)$ at $\infty$.  
\end{ex}

\begin{lm}
\label{LemmaTangentSpaceOfDilatation} Let $V\to X$ be a smooth morphism, let $Y\subset V_\infty$ be a smooth locally closed subscheme, and let $\pi\from V'\to V$ be the dilatation of $V$ at $Y$.  Then $V'\to X$ is smooth, and $\Tan_{V'/X}$ lies in an exact sequence of $\OO_{V'}$-modules
\begin{equation}
\label{EquationTanVXSequence}
 0 \to \Tan_{V'/X}\to \pi^*\Tan_{V/X} \to \pi^*j_*N_{Y/V_\infty}\to 0,
\end{equation}
where $N_{Y/V_\infty}$ is the normal bundle of $Y\subset V_\infty$, and $j\from Y\to V$ is the inclusion. 

Finally, let $T\to X$ be a morphism which is flat at $\infty$, and let $s\from T\to V$ be a morphism of $X$-schemes, such that $s_\infty$ factors through $Y$.  By the universal property of dilatations, $s$ factors through a morphism $s'\from T\to V'$.  Then we have an exact sequence of $\OO_V$-modules
\begin{equation}
\label{EquationTanVXSequencePullBack}
0 \to (s')^*\Tan_{V'/X} \to s^*\Tan_{V/X} \to i_{T_\infty*}s_\infty^*N_{Y/V_\infty} \to 0.
\end{equation}
\end{lm}

\begin{proof} One reduces to the case that $Y$ is closed in $V_\infty$. The smoothness of $V'\to X$ is \cite[\S 3.2, Proposition 3]{NeronModels}.  We turn to the exact sequence \eqref{EquationTanVXSequence}.  The morphism $\Tan_{V'/X}\to \pi^*\Tan_{V/X}$ comes from functoriality of the tangent bundle.  To construct the morphism $\pi^*\Tan_{V/X}\to \pi^*j_*N_{Y/V_\infty}$, we consider the diagram
\[
\xymatrix{
V_{\infty}'\ar[dd]_{i_{V'}} \ar[r]^{\pi_\infty'} \ar[dr]_{\pi_\infty} & Y \ar[d]^{i_Y} \ar@/^2.0pc/[dd]^j \\
& V_\infty \ar[d]^{i_V} \\
V' \ar[r]_{\pi} & V
}
\]
in which the outer rectangle is cartesian.  For its part, the normal bundle $N_{Y/V_\infty}$ sits in an exact sequence of $\OO_Y$-modules
\[ 0 \to \Tan_{Y/C}\to i_Y^*\Tan_{V_\infty/C} \to N_{Y/V_\infty} \to 0.\]
The composite
\begin{eqnarray*}
 i_{V'}^*\pi^*\Tan_{V/X} 
 &=&
 \pi_\infty^* i_V^* \Tan_{V/X}\\
 &\isom& \pi_\infty^* \Tan_{V_\infty/C} \\
 &=& (\pi_\infty')^* i_Y^* \Tan_{V/C} \\
 &\to& (\pi_\infty')^* N_{Y/V_\infty}
 \end{eqnarray*} 
 induces by adjunction a morphism 
 \[ \pi^*\Tan_{V/X}\to i_{V'*}(\pi_\infty')^* N_{Y/V_\infty} \isom \pi^*j_* N_{Y/V_\infty},\]
where the last step is justified because $j$ is a closed immersion.  

We check that \eqref{EquationTanVXSequence} is exact using our explicit description of $V'$.  The sequence is clearly exact away from the preimage of $Y$ in $V'$, since on the complement of this locus, the morphism $\pi$ is an isomorphism, and $\pi^*j_*=0$.  Therefore we let $y\in Y$ and check exactness after localization at $y$.   Let $\caI\subset \OO_V$ be the ideal sheaf which cuts out $Y$, and let $I\subset \OO_{V,y}$ be the localization of $\caI$ at $y$.  Then $\OO_{V_\infty,y}=\OO_{V,y}/\xi$.  Since $Y\subset V_\infty$ are both smooth at $y$, we can find a system of local coordinates $\overline{f}_1,\dots,\overline{f}_n\in \OO_{V_\infty,y}$ (meaning that the differentials $d\overline{f}_i$ form a basis for $\Omega^1_{V_{\infty}/C,y}$), such that $\overline{f}_{k+1},\dots,\overline{f}_n$ generate $I/\xi$.  If $\partial/\partial \overline{f}_i$ are the dual basis, then the stalk of $N_{Y/V_\infty}$ at $y$ is the free $\OO_{Y,y}$-module with basis $\partial/\partial\overline{f}_{k+1},\dots,\partial/\partial\overline{f}_n$.

Choose lifts $f_i\in \OO_{V,y}$ of the $\overline{f}_i$.  Then $I$ is generated by $\xi,f_k,\dots,f_n$.  The localization of $V'\to V$ over $y$ is $\Spec \OO_{V',y}$, where $\OO_{V',y}=\OO_{V,y}[g_{k+1},\dots,g_n]/(\xi\text{-torsion})$, where $\xi g_i=f_i$ for $i=k+1,\dots,n$.  Then the stalk of $\Tan_{V'/X}$ at $y$ is the free $\OO_{V',y}$-module with basis $\partial/\partial f_1,\dots,\partial/\partial f_k,\partial/\partial g_{k+1},\dots,\partial/\partial g_n$, whereas the stalk of $\pi^*\Tan_{V/X}$ at $y$ is the free $\OO_{V',y}$-module with basis $\partial/\partial f_1,\dots,\partial/ \partial f_n$.  The quotient between these stalks is evidently the free module over $\OO_{V',y}/\xi$ with basis $\partial/\partial f_{k+1},\dots,\partial /\partial f_n$, and this agrees with the stalk of $\pi^*j_*N_{Y/V_\infty}$.  

Given a morphism of $X$-schemes $s\from T\to V$ as in the lemma, we apply $(s')^*$ to \eqref{EquationTanVXSequence};  this is exact because $s'$ is flat.  The term on the right is $s^*j_*N_{Y/V_\infty}\isom i_{T_\infty *}s_\infty^* N_{Y/V_\infty}$ (once again, this is valid because $j$ is a closed immersion).
\end{proof}

\subsection{The space $\cM_{H,\infty}$ as global sections of a scheme over $X_C$}
\label{SectionTheSpacesMH}

We will prove Theorem \ref{MainTheorem} for the Rapoport-Zink spaces of the form $\cM_{H,\infty}$ before proceeding to the general case.  Let $H$ be a $p$-divisible group of height $n$ and dimension $d$ over a perfect field $k$.  In this context, $\breve{E}=W(k)[1/p]$.  Let $\caE=\caE_C(H)$.  Throughout, we will be interpreting $\cM_{H,\infty}$ as a functor on $\Perf_{\breve{E}}$ as in Proposition \ref{DescriptionOfMH}.  

We have a determinant morphism $\det\from \cM_{H,\infty}\to \cM_{\det H, \infty}$.   Let $\tau\in \cM_{\det H,\infty}(C)$ be a geometric point of $\cM_{\det H, \infty}$.  This point corresponds to a section $\tau$ of $\bV(\det\caE)\to X_C$, which we also call $\tau$.  Let $\cM_{H,\infty}^{\tau}$ be the fiber of $\det$ over $\tau$.  

Our first order of business is to express $\cM_{H,\infty}^\tau$ as the space of global sections of a smooth morphism $Z\to X_C$, defined as follows.  We have the geometric vector bundle $\bV(\caE^n)\to X$, whose global sections parametrize morphisms $s\from \OO_{X_C}^n\to \caE$. Let $U_{n-d}$ be the locally closed subscheme of the fiber of $\bV(\caE^n)$ over $\infty$, which parametrizes all morphisms of rank $n-d$. We consider the dilatation $\bV(\caE^n)^{\rk_\infty = n-d} \rar \bV(\caE^n)$ of $\bV(\caE^n)$ along $U_{n-d}$.  For any flat $X_C$-scheme $T$, $\bV(\caE^n)^{\rk_\infty = n-d}(T)$ is the set of all $s \colon \caO_T^n \rar \caE_T$ such that ${\rm cok}(s) \otimes C$  is projective $\caO_T \otimes C$-module of rank $d$.
Define $Z$ as the Cartesian product:
\begin{equation}
\label{EquationDiagramZ}
\xymatrix{
Z \ar[d] \ar[r] &
\bV(\caE^n)^{\rk_\infty = n-d} \ar[d]^{\det} \\
X_C \ar[r]_(0.4){\tau} & \bV(\det \caE).
}
\end{equation}

\begin{lm}\label{lm:MZ_and_MHinfty}  Let  $\cM_Z$ be the functor which inputs a perfectoid space $T/C$ and outputs the set of sections of $Z\to X_C$ over $X_T$.  Then $\cM_Z$ is isomorphic to $\cM_{H,\infty}^\tau$.  
\end{lm}

\begin{proof}  Let $T=\Spa(R,R^+)$ be an affinoid perfectoid space over $C$.  The morphism $X_T\to X_C$ is flat.  (This can be checked locally:  $B_{\dR}^+(R)$ is torsion-free over the discrete valuation ring $B_{\dR}^+(C)$, and so it is flat.)  By the description in \eqref{EquationDiagramZ}, an $X_T$-point of $\cM_Z$ corresponds to a morphism $\sigma\from \OO_{X_T}^n\to \caE_T(H)$ which has the properties:
\begin{enumerate}
\item[(1)] The cokernel of $\sigma_\infty$ is a projective $R$-module quotient of $\caE_T(H)_\infty$ of rank $d$.
\item[(2)] The determinant of $\sigma$ equals $\tau$.  
\end{enumerate}
On the other hand, by Proposition \ref{DescriptionOfMH}, $\cM_{H,\infty}(T)$ is the set of morphisms $\sigma\from \OO_{X_T}^n\to \caE_T(H)$ satisfying
\begin{enumerate}
\item[($1'$)] The cokernel of $\sigma$ is $i_{\infty*} W$, for a projective $R$-module quotient $W$ of $\caE_T(H)_\infty$ of rank $d$.
\item[($2$)] The determinant of $\sigma$ equals $\tau$.
\end{enumerate}
We claim the two sets of conditions are equivalent for a morphism $\sigma\from \OO_{X_T}^n\to \caE_T(H)$.  
Clearly ($1'$) implies (1), so that ($1'$) and ($2$) together imply (1) and (2) together.  Conversely, suppose (1) and (2) hold.  Since $\tau$ represents a point of $\cM_{\det H,\infty}$, it is an isomorphism outside of $\infty$, and therefore so is $\sigma$.  This means that $\cok \sigma$ is supported at $\infty$.  Thus $\cok \sigma$ is a $B_{\dR}^+(R)$-module.  For degree reasons, the length of $(\cok \sigma)\otimes_{B_{\dR}^+(R)} B_{\dR}^+(C')$ has length $d$ for every geometric point $\Spa(C',(C')^+)\to T$.    Whereas condition (1) says that $(\cok \sigma)\otimes_{B_{\dR}^+(R)} R$ is a projective $R$-module of rank $d$.   This shows that $(\cok \sigma)$ is already a projective $R$-module of rank $d$, which is condition ($1'$).  
\end{proof}

\begin{lm}
The morphism $Z\to X_C$ is smooth. 
\end{lm}

\begin{proof}

Let $\infty'\in X_C$ be a closed point, with residue field $C'$.  It suffices to show that the stalk of $Z$ at $\infty'$ is smooth over $\Spec B_{\dR}^+(C')$.

If $\infty'\neq \infty$, then this stalk is isomorphic to the variety $(\bA^{n^2})^{\det = \tau}$ consisting of $n\times n$ matrices with fixed determinant $\tau$.  Since $\tau$ is invertible in $B_{\dR}^+(C')$, this variety is smooth.

Now suppose $\infty'=\infty$.  Let $\xi$ be a generator for the kernel of $B_{\dR}^+(C)\to C$.  Then the stalk of $Z$ at $\infty$ is isomorphic to the flat $B_{\dR}^+(C)$-scheme $Y$, whose $T$-points for a flat $B_{\dR}^+(C)$-scheme $T$ are $n\times n$ matrices with coefficients in $\Gamma(T,\OO_T)$, which are rank $n-d$ modulo $\xi$, and which have fixed determinant $\tau$ (which must equal $u\xi^d$ for a unit $u\in B_{\dR}^+(C)$).  Consider the open subset $Y_0\subset Y$ consisting of matrices $M$ where the first $(n-d)$ columns have rank $(n-d)$.  Then the final $d$ columns of $M$ are congruent modulo $\xi$ to a linear combination of the first $(n-d)$ columns.  After row reduction operations only depending on those first $(n-d)$ columns, $M$ becomes
\[
\left(
\begin{array}{c|c}
I_{n-d} & P \\
\hline
0 & \xi Q
\end{array}
\right),
\]
with $\det Q=w$ for a unit $w\in B_{\dR}^+(C)$ which only depends on the first $(n-d)$ columns of $M$.  We therefore have a fibration $Y_0\to \bA^{n(n-d)}$, namely projection onto the first $(n-d)$ columns, whose fibers are $\bA^{d(n-d)} \times (\bA^{d^2})^{\det = w}$, which is smooth.  Therefore $Y_0$ is smooth.  The variety $Y$ is covered by opens isomorphic to $Y_0$, and so it is smooth.
\end{proof}

We intend to apply Theorem \ref{TheoremCohomologicalSmoothnessCriterion} to the morphism $Z\to X$, and so we need some preparations regarding the relative tangent space of $\bV(\caE^n)^{\rk_\infty= n-d}\to X_C$.

\subsection{A linear algebra lemma}
Let $f\from V'\to V$ be a rank $r$ linear map between $n$-dimensional vector spaces over a field $C$. Thus there is an exact sequence
\[ 0 \to W'\to V' \stackrel{f}{\to} V \stackrel{q}{\to} W \to 0.\]
with $\dim W=\dim W'=n-r$. 

Consider the minor map $\Lambda\from \Hom(V',V)\to \Hom(\bigwedge^{r+1}V',\bigwedge^{r+1}V)$ given by $\sigma\mapsto \bigwedge^{r+1} \sigma$.  This is a polynomial map, whose derivative at $f$ is a linear map
\[
\label{EquationDfF}
D_f\Lambda\from \Hom(V',V) \to \Hom\left(\bigwedge^{r+1}V',\bigwedge^{r+1} V\right).\]
Explicitly, this map is 
\[ 
\label{EquationDfF}
D_f\Lambda(\sigma)(v_1\wedge\cdots \wedge v_{r+1}) =
\sum_{i=1}^{r+1} f(v_1)\wedge f(v_2)\wedge \cdots \wedge \sigma(v_i) \wedge \cdots \wedge f(v_{r+1}).\]

\begin{lm} \label{LemmaMinorMap} Let 
\[ K=\ker\left(\Hom(V',V) \to \Hom(W',W)\right)\]
be the kernel of the map $\sigma\mapsto q\circ (\sigma\vert_{W'})$.
Then $\ker D_f\Lambda = K$.
\end{lm}

\begin{proof}  Suppose $\sigma\in K$.  Since $f$ has rank $r$, the exterior power $\bigwedge^{r+1} V'$ is spanned over $C$ by elements of the form $v_1\wedge\cdots\wedge v_{r+1}$, where $v_{r+1}\in \ker f = W'$.  Since $f(v_{r+1})=0$, the sum in \eqref{EquationDfF} reduces to 
\[ D_f\Lambda(\sigma)(v_1\wedge\cdots\wedge v_{r+1}) = f(v_1)\wedge \cdots \wedge f(v_r) \wedge \sigma(v_{r+1}). \]
Since $\sigma\in K$ and $v_{r+1}\in W'$ we have $\sigma(v_{r+1})\in \ker q=f(V')$, which means that $D_f\Lambda(\sigma)(v_1,\dots,v_{r+1})\in \bigwedge^{r+1}f(V')=0$.  Thus $\sigma\in \ker D_f\Lambda$.

Now suppose $\sigma\in \ker D_f\Lambda$.   Let $w\in W'$.  We wish to show that $\sigma(w)\in f(V')$. Let $v_1,\dots,v_r\in V'$ be vectors for which $f(v_1),\cdots,f(v_r)$ is a basis for $f(V')$.  Since $\sigma\in \ker D_f\Lambda$, we have $D_f\Lambda(\sigma)(v_1\wedge\cdots\wedge v_r\wedge w)=0$.  On the other hand, 
\[ D_f\Lambda(\sigma)(v_1\wedge \cdots\wedge v_r\wedge w) = f(v_1)\wedge\cdots \wedge f(v_r) \wedge \sigma(w), \]
because all other terms in the sum in \eqref{EquationDfF} are 0, owing to $f(w)=0$.  Since the wedge product above is 0, and the $f(v_i)$ are a basis for $f(V')$, we must have $\sigma(w)\in f(V')$.  Thus $\sigma\in K$.
\end{proof}

We interpret Lemma \ref{LemmaMinorMap} as the calculation of a certain normal bundle.  Let $Y=\bV(\Hom(V',V))$ be the affine space over $C$ representing morphisms $V'\to V$ over a $C$-scheme, and let $j\from Y^{\rk =r}\to Y$ be the locally closed subscheme representing morphisms which are everywhere of rank $r$.   Thus, $Y^{\rk = r}$ is an open subset of the fiber over 0 of (the geometric version of) the minor map $\Lambda$.  It is well known that $Y^{\rk = r}/C$ is smooth of codimension $(n-r)^2$ in $Y/C$, and so the normal bundle $N_{Y^{\rk= r}/Y}$ is locally free of this rank.  

We have a universal morphism of $\OO_{Y^{\rk=r}}$-modules $\sigma\from \OO_{Y^{\rk = r}}\otimes_C V'\to \OO_{Y^{\rk = r}}\otimes_C V$.  Let  $\caW'=\ker \sigma$ and $\caW=\cok \sigma$, so that $\caW'$ and $\caW$ are locally free $\OO_{Y^{\rk = r}}$-modules of rank $n-r$. We also have the $\OO_{Y^{\rk = r}}$-linear morphism $D\Lambda \colon \OO_{Y^{\rk=r}} \otimes_C \Hom(V',V) \rar \OO_{Y^{\rk=r}} \otimes_C \Hom(\Lambda^{r+1}V',\Lambda^{r+1}V)$, whose kernel is precisely $\Tan_{Y^{\rk=r}/C}$. The geometric interpretation of Lemma \ref{LemmaMinorMap} is a commutative diagram with short exact rows:
\begin{equation}
\label{EquationNormalBundle}
\xymatrix{
\ker D\Lambda \ar[r] \ar[d]_{\isom} & \OO_{Y^{\rk=r}} \otimes_C \Hom(V',V) \ar[r] \ar[d]_{\isom} & 
\SHom(\caW',\caW)
\ar[d]^{\isom} \\
\Tan_{Y^{\rk=r}/C} \ar[r]& j^*\Tan_{Y/C} \ar[r] & N_{Y^{\rk=r}/Y}.
}
\end{equation}

\subsection{Moduli of morphisms of vector bundles with fixed rank at $\infty$}
We return to the setup of \S\ref{SectionDilatations}.  We have a curve $X$ and a closed point $\infty\in X$, with inclusion map $i_\infty$ and residue field $C$.

Let $\caE$ and $\caE'$ be rank $n$ vector bundles over $X$, with fibers $V=\caE_\infty$ and $V'=\caE'_\infty$.   We have the geometric vector bundle $\bV(\SHom(\caE',\caE))\to X$.  If $f\from T\to X$ is a morphism, then $T$-points of $\bV(\SHom(\caE',\caE))$ classify $\OO_T$-linear maps $f^*\caE'\to f^*\caE$.  

Let $\bV(\SHom(\caE',\caE))^{\rk_\infty = r}$ be the dilatation of $\bV(\SHom(\caE',\caE))$ at the locally closed subscheme $\bV(\Hom(V',V))^{\rk =r}$ of the fiber $\bV(\SHom(\caE',\caE))_\infty=\bV(\Hom(V',V))$.  This has the following property, for a flat morphism $f\from T\to X$:  the $X$-morphisms
$s\from T\to\bV(\SHom(\caE',\caE))^{\rk_\infty = r}$ parametrize those $\OO_T$-linear maps $\sigma\from f^*\caE'\to f^*\caE$, for which 
the fiber $\sigma_\infty\from f_\infty^*V'\to f_\infty^*V$ has rank $r$ everywhere on $T_\infty$.

Given a morphism $s$ as above, corresponding to a morphism $\sigma\from f^*\caE'\to f^*\caE$, we let $\caW'$ and $\caW$ denote the kernel and cokernel of $\sigma_\infty$.  Then $\caW'$ and $\caW$ are locally free $\OO_{T_\infty}$-modules of rank $r$.  Let $i_{T_\infty}\from T_\infty\to T$ denote the pullback of $i_\infty$ through $f$.  

We intend to use Lemma \ref{LemmaTangentSpaceOfDilatation} to compute
$s^*\Tan_{\bV(\SHom(\caE',\caE))^{\rk_\infty =r}/X}$.  The tangent bundle $\Tan_{\bV(\SHom(\caE',\caE))/X}$ is isomorphic to the pullback $f^*\SHom(\caE',\caE)$.  Also, we have identified the normal bundle $N_{\bV(\Hom(V',V))^{\rk =r}/\bV(\Hom(V',V)}$ in \eqref{EquationNormalBundle}.  So when we apply the lemma to this situation, we obtain an exact sequence of $\OO_T$-modules
\begin{equation}
\label{EquationIdentificationOfKerLambda}
 0 \to s^*\mathrm{Tan}_{\bV(\SHom(\caE',\caE))^{\rk_\infty= r}/X} \to f^*\SHom(\caE',\caE) \to i_{T_\infty *}\SHom(\caW', \caW) \to 0,
\end{equation}
where the third arrow is adjoint to the map 
\[ i_{T_\infty}^*f^*\SHom(\caE',\caE)=
\Hom(f_\infty^*V',f_\infty^*  V)\to \SHom(\caW',\caW), \]
which sends $\sigma\in\SHom(f_\infty^*V',f_\infty^*V)$ to the composite
\[ \caW'\to f_\infty^*V' \stackrel{\sigma_\infty}{\to}f_\infty^*V \to \caW.\]

The short exact sequence in \eqref{EquationIdentificationOfKerLambda}
identifies the $\OO_T$-module $s^*\mathrm{Tan}_{\bV(\SHom(\caE',\caE))^{\rk_\infty= r}/X}$ as a modification of $f^*\SHom(\caE',\caE)$ at the divisor $T_\infty$.  We can say a little more in the case that $\sigma$ itself is a modification.  Let us assume that $\sigma$ fits into an exact sequence
\[ 0 \to f^\ast\caE'\stackrel{\sigma}{\to} f^\ast\caE \stackrel{\alpha}{\to} i_{T_\infty *}\caW \to 0.\]  Dualizing gives another exact sequence
\[ 0 \to f^\ast(\caE^{\vee}) \stackrel{\sigma^\vee}{\to} f^\ast(\caE')^\vee \stackrel{\alpha'}{\to} i_{T_\infty*} (\caW')^\vee \to 0.\]
Then
\begin{eqnarray*}
s^*\mathrm{Tan}_{\bV(\SHom(\caE',\caE))^{\rk_\infty= r}/X} &=& \ker\left[
f^\ast\SHom(\caE',\caE) \to i_{T_\infty*}\SHom(\caW',\caW) \right]\\
&\cong& \ker(\alpha\otimes \alpha')
\end{eqnarray*}
The kernel of $\alpha\otimes\alpha'$ can be computed in terms of $\ker \alpha=f^\ast\caE'$ and $\ker \alpha'=f^\ast(\caE^\vee)$, see Lemma \ref{LemmaTor} below.  It sits in a diagram

\begin{equation}
\label{EquationSESker}
\xymatrix{
& 0 \ar[d] &&& \\
& f^\ast\SHom(\caE,\caE') \ar[d] &&&\\
0 \ar[r] & \caF \ar[r] \ar[d] &f^\ast\SHom(\caE,\caE)\oplus f^\ast\SHom(\caE',\caE') \ar[r] & s^*\mathrm{Tan}_{\bV(\SHom(\caE',\caE))^{\rk_\infty = r}/X} \ar[r] & 0.\\
& \Tor_1(i_{\infty *} \caW',i_{\infty *} \caW) \ar[d] &&&\\
& 0 &&&
}
\end{equation}

\begin{lm} \label{LemmaTor} Let $\caA$ be an abelian $\otimes$-category.  Let 
\[ 0 \to K\stackrel{i}{\to} A \stackrel{f}{\to} B \to 0 \]
\[ 0 \to K' \stackrel{i'}{\to} A' \stackrel{f'}{\to} B'\to 0\]
be two exact sequences in $\caA$, with $A,A',K,K'$ projective.  The homology of the complex
\[
\xymatrix{
 K\otimes K' \ar[rr]^{(i\otimes 1_{K'},1_K\otimes i')\kern 1cm} &\;\;\; & (A\otimes K')\oplus (K\otimes A) \ar[rr]^{\kern 1cm 1_A \otimes i' - i \otimes 1_{A'} } &\;\;\;& A \otimes A' }
 \]
is given by $H_2=0$, $H_1\cong \Tor_1(B,B')$, and $H_0\cong B\otimes B'$.  Thus, $K''=\ker(f\otimes f'\from A\otimes A' \to B\otimes B')$ appears in a diagram
\[
\xymatrix{
& 0 \ar[d] &&&\\
& K\otimes K' \ar[d] & & & \\
0\ar[r] & L\ar[r]\ar[d]  & (A\otimes K')\oplus (K\otimes A) \ar[r] & K'' \ar[r]& 0 \\
& \Tor_1(B,B') \ar[d] &&&\\
& 0&&&
}
\]
where both sequences are exact.
\end{lm}

\begin{proof} Let $C_{\bullet}$ be the complex $K\to A$, and let $C'_{\bullet}$ be the complex $K'\to A'$.  
Since $C'_{\bullet}$ is a projective resolution of $B'$, we have a Tor spectral sequence \cite[\href{https://stacks.math.columbia.edu/tag/061Z}{Tag 061Z}]{StacksProject} 
\[ E^2_{i,j}\from \Tor_j(H_i(C_{\bullet}),B') \implies H_{i+j}(C_{\bullet}\otimes C_{\bullet}'). \]
We have $E^2_{0,0}=B\otimes B'$ and $E^2_{0,1}=\Tor_1(B,B')$, and $E^2_{i,j}=0$ for all other $(i,j)$.  Therefore $H_0(C_{\bullet} \otimes C'_{\bullet}) \cong B\otimes B'$ and $H_1(C_{\bullet} \otimes C'_{\bullet}) \cong \Tor_1(B,B')$, which is the lemma.
\end{proof} 

\subsection{A tangent space calculation}  We return to the setup of \S\ref{SectionTheSpacesMH}.  Thus we have fixed a $p$-divisible group $H$ over a perfect field $k$, and an algebraically closed perfectoid field $C$ containing $W(k)[1/p]$.   But now we specialize to the case that $H$ is isoclinic.  Therefore $D=\End H$ (up to isogeny) is a central simple $\Q_p$-algebra.  Let $\caE=\caE_C(H)$;  we have $\SHom(\caE,\caE)\isom D\otimes_{\Q_p} \OO_{X_C}$.   

Recall the scheme $Z\to X_C$, defined as a fiber product in \eqref{EquationDiagramZ}.  Let $s\from X_C\to Z$ be a section.  This corresponds to a morphism $\sigma\from \OO_{X_C}^n\to\caE$.  Let $W'$ and $W$ be the cokernel of $\sigma_\infty$;  these are $C$-vector spaces. 

We are interested in the vector bundle $s^*\mathrm{Tan}_{Z/X_C}$.   This is the kernel of the derivative of the determinant map:
\[ s^*\mathrm{Tan}_{Z/X_C}=\ker\left(D_s\det\from s^*\mathrm{Tan}_{\bV(\caE^n)^{\rk_\infty= n-d}/X_C}\to \det \caE\right).\]
We apply \eqref{EquationSESker} to give a description of $s^*\mathrm{Tan}_{\bV(\caE^n)^{\rk_\infty= n-d}/X_C}$.  We get a diagram of $\OO_{X_C}$-modules
\begin{equation}
\label{EquationSESker2}
\xymatrix{
& 0 \ar[d] &&& \\
& (\caE^{\vee})^n \ar[d] &&&\\
0 \ar[r] & \caF \ar[r] \ar[d] & (M_n(\Q_p)\times D)\otimes \OO_{X_C}  \ar[r] & s^*\mathrm{Tan}_{\bV(\mathcal{E}^n)^{\rk_\infty= n-d}/X_C} \ar[r] & 0.\\
& \Tor_1(i_{\infty *} W',i_{\infty *} W) \ar[d] &&&\\
& 0 &&&
}
\end{equation}
On the other hand, the horizontal exact sequence fits into a diagram
\begin{equation}
\label{DescriptionOfTanV}
\xymatrix{
0 \ar[r] & \caF\ar[r] & (M_n(\Q_p)\times D)\otimes \OO_{X_C} \ar[d]_{\tr} \ar[r] & s^*\mathrm{Tan}_{\bV(\caE^n)^{\rk_\infty\leq n-d}/X_C} \ar[r] \ar[d]^{D_s\det} & 0 \\
& & \OO_{X_C} \ar[r]_{\tau} & \det \caE & 
}
\end{equation}

The arrow labeled $\tr$ is induced from the $\Q_p$-linear map $M_n(\Q_p)\times D\to \Q_p$ carrying $(\alpha',\alpha)$ to $\tr(\alpha')-\tr(\alpha)$ (reduced trace on $D$).  The commutativity of the lower right square boils down to the identity, valid for sections $s_1,\dots,s_n\in H^0(X_C,\caE)$ and $\alpha\in D$:
\[
\left((\alpha s_1)\wedge s_2\wedge\cdots \wedge s_n\right) +\cdots +
\left(s_1\wedge \cdots \wedge (\alpha s_n)\right)
=(\tr\; \alpha) (s_1\wedge\cdots\wedge s_n). \]
(There is a similar identity for $\alpha'\in M_n(\Q_p)$.)
Because the arrow labeled $\tau$ is injective, we can combine \eqref{EquationSESker2} and \eqref{DescriptionOfTanV} to arrive at a description of $s^*\Tan_{Z/X_C}$:
\begin{equation}
\label{DescriptionOfTanZ}
\xymatrix{
& 0 \ar[d] &&& \\
& (\caE^\vee)^n \ar[d] &&&\\
0 \ar[r] & \caF \ar[r] \ar[d] & (M_n(\Q_p)\times D)^{\tr =0}\otimes \OO_X  \ar[r] & s^*\mathrm{Tan}_{Z/X_C} \ar[r] & 0.\\
& \Tor_1(i_{\infty *} W',i_{\infty *} W) \ar[d] &&&\\
& 0 &&&
}
\end{equation}

We pass to duals to obtain
\begin{equation}
\label{DescriptionOfTanZDual}
\xymatrix{
& & & 0 \ar[d] &\\
0 \ar[r] &
(s^*\mathrm{Tan}_{Z/X_C})^\vee \ar[r] &
((M_n(\Q_p)\times D)/\Q_p) \otimes_{\Q_p} \OO_{X_C}\ar@{..>}[dr]\ar[r]& \caF^\vee\ar[d] \ar[r] & 0
\\
&&& \caE^n \ar[d] & \\
& & & \Tor_1(i_{\infty *}((W')^\vee,i_{\infty*}W^\vee) \ar[d] &\\
& & & 0 &
}
\end{equation}
The dotted arrow is induced from the map $(M_n(\Q_p)\times D)\otimes_{\Q_p} \OO_{X_C}\to \caE^n$ sending $(\alpha',\alpha)\otimes 1$ to $\alpha\circ \sigma- \sigma\circ \alpha'$.  
\begin{thm}
\label{TheoremHNSlopes}
If $s$ is a section to $Z \rightarrow X_C$ corresponding, under the  isomorphism of Lemma \ref{lm:MZ_and_MHinfty}, to a point $x \in \caM_{H,\infty}^\tau(C)$, then the following are equivalent: 
\begin{enumerate}
\item The vector bundle $s^*\mathrm{Tan}_{Z/X_C}$ has a Harder-Narasimhan slope which is $\leq 0$.
\item The point $x$ lies in the special locus $\cM_{H,\infty}^{\tau,\spe}$.
\end{enumerate}
\end{thm}

\begin{proof} Let $\sigma\from \OO_{X_C}^n\to \caE$ denote the homomorphism corresponding to $x$. Condition (1) is true if and only if $H^0(X_C,s^*\mathrm{Tan}_{Z/X_C}^\vee)\neq 0$.  We now take $H^0$ of \eqref{DescriptionOfTanZDual}, noting that $H^0(X_C,\caF^\vee)\to H^0(X_C,\caE^n)$ is injective.   We find that
\begin{eqnarray*}
 H^0(X_C,s^*\mathrm{Tan}_{Z/X_C}^\vee)&\cong& \left\{(\alpha',\alpha)\in M_n(\Q_p)\times D\biggm\vert\alpha\circ \sigma=\sigma\circ \alpha'\right\}/\Q_p. \\
 &=& A_x/\Q_p.
 \end{eqnarray*}
This is nonzero exactly when $x$ lies in the special locus.
\end{proof}

Combining Theorem \ref{TheoremHNSlopes} with the criterion for cohomological smoothness in Theorem \ref{TheoremCohomologicalSmoothnessCriterion} proves Theorem \ref{MainTheorem} for the space $\cM_{H,\infty}$.

Naturally we wonder whether it is possible to give a complete discription of  $s^*\Tan_{Z/X_C}$, as this is the ``tangent space'' of $\cM_{H,\infty}^\tau$ at the point $x$.  Note that $s^*\Tan_{Z/X_C}$ can only have nonnegative slopes, since it is a quotient of a trivial bundle.  Therefore Theorem \ref{TheoremHNSlopes} says that 0 appears as a slope of $s^*\Tan_{Z/X_C}$ if and only if $s$ corresponds to a special point of $\cM_{H,\infty}^\tau$.  

\begin{ex} Consider the case that $H$ has dimension 1 and height $n$, so that $\cM_{H,\infty}$ is an infinite-level Lubin-Tate space.  Suppose that $x\in \cM_{H,\infty}(C)$ corresponds to a section $s\from X_C\to Z$.  Then $s^*\Tan_{Z/X_C}$ is a vector bundle of rank $n^2-1$ and degree $n-1$, with slopes lying in $[0,1/n]$;  this already limits the possibilities for the slopes to a finite list.

If $n=2$ there are only two possibilities for the slopes appearing in $s^*\Tan_{Z/X_C}$:  $\set{1/3}$ and $\set{0,1/2}$.  These correspond exactly to the nonspecial and special loci, respectively.

If $n=3$, there are a priori five possibilities for the slopes appearing in $s^*\Tan_{Z/X_C}$:  $\set{1/4,1/4}$, $\set{1/3,1/5}$, $\set{1/3,1/3,0,0}$, $\set{2/7,0}$, and $\set{1/3,1/4,0}$.  But in fact the final two cases cannot occur:  if 0 appears as a slope, then $x$ lies in the special locus, so that $A_x\neq \Q_p$.  But as $A_x$ is isomorphic to a subalgebra of $\End^\circ H$, the division algebra of invariant 1/3, it must be the case that $\dim_{\Q_p} A_x = 3$, which forces 0 to appear as a slope with multiplicity $\dim_{\Q_p} A_x/\Q_p = 2$.  On the nonspecial locus, we suspect that the generic (semistable) case $\set{1/4,1/4}$ always occurs, as otherwise there would be some unexpected stratification of $\cM_{H,\infty}^{\circ,\ns}$.  But currently we do not know how to rule out the case $\set{1/3,1/5}$.
\end{ex}

\subsection{The general case}  Let $\caD = (B,V,H,\mu)$ be a rational EL datum over $k$, with reflex field $E$.  Let $F$ be the center of $B$.  As in Section \ref{SpecialLocus}, let $D=\End_{B} V$ and $D'=\End_{B} H$, so that $D$ and $D'$ are both $F$-algebras.  

 Let $C$ be a perfectoid field containing $\breve{E}$, and let $\tau\in \cM_{\det \caD,\infty}(C)$.  Let $\cM_{\caD,\infty}^\tau$ be the fiber of the determinant map over $\tau$.  We will sketch the proof that $\cM_{\caD,\infty}^\tau\to \Spa C$ is cohomologically smooth.  It is along the same lines as the proof for $\cM_{H,\infty}$, but with some extra linear algebra added.

The space $\cM_{\caD,\infty}^{\tau}$ may be expressed as the space of global sections of a smooth morphism $Z\to X_C$, defined as follows.  We have the geometric vector bundle $\bV(\SHom_B(V\otimes_{\Q_p}\OO_X, \caE_C(H)))$. In its fiber over $\infty$, we have the locally closed subscheme whose $R$-points for a $C$-algebra $R$ are morphisms, whose cokernel is as a $B\otimes_{\bQ_p} R$-module isomorphic to $V_0 \otimes_{\breve E} R$, where $V_0$ is the weight 0 subspace of $V\otimes_{\Q_p} \breve{E}$ determined by $\mu$. We then have the dilatation $\bV(\SHom_B(V\otimes_{\Q_p}\OO_{X_C}, \caE_C(H)))^{\mu}$ of $\bV(\SHom_B(V\otimes_{\Q_p}\OO_X, \caE_C(H)))$ at this locally closed subscheme. Its points over $S=\Spa(R,R^+)$ parametrize $B$-linear morphisms $s\from V\otimes_{\Q_p} \OO_{X_S}\to \caE_S(H)$, such that (locally on $S$) the cokernel of the fiber $s_\infty$ is isomorphic as a $(B\otimes_{\Q_p} R)$-module to $V_0\otimes_{\breve{E}} R$. Finally, the morphism $Z\to X_C$ is defined by the cartesian diagram
\[
\xymatrix{
Z \ar[r] \ar[d] & \bV(\SHom_B(V\otimes_{\Q_p}\OO_{X_C}, \caE_C(H)))^{\mu}  \ar[d]^{\det} \\
X_C \ar[r]_(.17){\tau} & \bV(\SHom_F(\det_F V\otimes_{\Q_p}\OO_{X_C}, \det_F \caE_C(H))).
}
\]
%

Let $x\in \cM_{\caD,\infty}(C)$ correspond to a $B$-linear morphism $s\from V\otimes_{\Q_p} \OO_{X_C} \to \caE_C(H)$ and a section of 
$Z\to X_C$ which we also call $s$.  Define $B\otimes_{\Q_p} C$-modules $W'$ and $W$ by 
\[ 0 \to W'\to V\otimes_{\Q_p} C \stackrel{s_\infty}{\to} \caE_C(H)_\infty \to W \to 0.\]
The analogue of \eqref{DescriptionOfTanZDual} is a diagram which computes the dual of $s^*\Tan_{Z/X_C}$:
\begin{equation}
\label{DescriptionOfTanZDualGeneralized}
\xymatrix{
& & & 0 \ar[d] &\\
0 \ar[r] &
(s^*\mathrm{Tan}_{Z/X_C})^\vee \ar[r] &
((D'\times D)/F) \otimes_{\Q_p} \OO_{X_C}\ar@{..>}[dr]\ar[r]& \caF^\vee\ar[d] \ar[r] & 0
\\
&&& \SHom(V\otimes_{\Q_p} \OO_{X_C},\caE_C(H)) \ar[d] & \\
& & & \Tor_1^F(i_{\infty *}((W')^\vee,i_{\infty*}W^\vee) \ar[d] &\\
& & & 0 &
}
\end{equation}
This time, the dotted arrow is induced from the map $(D'\times D)\otimes_{\Q_p} \OO_{X_C}\to \SHom(V\otimes_{\Q_p} \OO_{X_C},\caE_C(H))$ sending $(\alpha',\alpha)\otimes 1$ to $\alpha\circ s- s\circ \alpha'$.   Taking $H^0$ in \eqref{DescriptionOfTanZDualGeneralized} shows that $H^0(X_C,s^*\Tan_{Z/X_C}^\vee) = A_x / F$, and this is nonzero exactly when $x$ lies in the special locus.  

\subsection{Proof of Corollary \ref{CorollaryModularCurve}}

We conclude with a discussion of the infinite-level modular curve $X(p^\infty)$.  Recall from \cite{ScholzeTorsion} the following facts about the Hodge-Tate period map $\pi_{HT}\from X(p^\infty)\to \bP^1$.  The ordinary locus in $X(p^\infty)$ is sent to $\bP^1(\Q_p)$.  The supersingular locus is isomorphic to finitely many copies of $\cM_{H,\infty,C}$, where $H$ is a connected $p$-divisible group of height 2 and dimension 1 over the residue field of $C$;  the restriction of $\pi_{HT}$ to this locus agrees with the $\pi_{HT}$ we had already defined on each $\cM_{H,\infty,C}$. 

We claim that the following are equivalent for a $C$-point $x$ of $X(p^\infty)^\circ$:
\begin{enumerate}
\item The point $x$ corresponds to an elliptic curve $E/\OO_C$, such that the $p$-divisible group $E[p^\infty]$ has $\End E[p^\infty]=\Z_p$.  
\item The stabilizer of $\pi_{HT}(x)$ in $\PGL_2(\Q_p)$ is trivial.
\item There is a neighborhood of $x$ in $X(p^\infty)^\circ$ which is cohomologically smooth over $C$.
\end{enumerate}

First we discuss the equivalence of (1) and (2).  If $E$ is ordinary, then  $E[p^\infty]\isom \Q_p/\Z_p \times \mu_{p^\infty}$ certainly has endomorphism ring larger than $\Z_p$, so that (1) is false.   Meanwhile, the stabilizer of $\pi_{HT}(x)$ in $\PGL_2(\Q_p)$ is a Borel subgroup, so that (2) is false as well.  The equivalence between (1) and (2) in the supersingular case is a special case of the equivalence discussed in Section \ref{SpecialLocus}.  

Theorem \ref{MainTheorem} tells us that $\cM_{H,\infty}^{\circ,\ns}$ is cohomologically smooth, which implies that shows that (2) implies (3).  We therefore are left with showing that if (2) is false for a point $x\in X(p^\infty)^{\circ}$, then no neighborhood of $x$ is cohomologically smooth.

First suppose that $x$ lies in the ordinary locus.  This locus is fibered over $\bP^1(\Q_p)$.  Suppose $U$ is a sufficiently small neighborhood of $x$.  Then $U$ is contained in the ordinary locus, and so $\pi_0(U)$ is nondiscrete.  This implies that $H^0(U,\bF_\ell)$ is infinite, and so $U$ cannot be cohomologically smooth. 

Now suppose that $x$ lies in the supersingular locus, and that $\pi_{HT}(x)$ has nontrivial stabilizer in $\PGL_2(\Q_p)$.  We can identify $x$ with a point in $\cM_{H,\infty}^{\circ,\spe}(C)$.  We intend to show that every neighborhood of $x$ in $\cM_{H,\infty}^{\circ}$ fails to be cohomologically smooth.

Not knowing a direct method, we appeal to the calculations in \cite{WeinsteinSemistable}, which constructed semistable formal models for each $\cM_{H,m}^{\circ}$.  The main result we need is Theorem 5.1.2, which uses the term ``CM points'' for what we have called special points.   There exists a decreasing basis of neighborhoods $Z_{x,0}\supset Z_{x,1}\supset\cdots$ of $x$ in $\cM_{H,\infty}^{\circ}$.  For each affinoid $Z=\Spa(R,R^+)$, let $\overline{Z}=\Spec R^+\otimes_{\OO_C} \kappa$, where $\kappa$ is the residue field of $C$.  For each $m\geq 0$, there exists a nonconstant morphism $\overline{Z}_{x,m}\to C_{x,m}$, where $C_{x,m}$ is an explicit nonsingular affine curve over $\kappa$.  This morphism is equivariant for the action of the stabilizer of $Z_{x,m}$ in $\SL_2(\Q_p)$.  For infinitely many $m$, the completion $C_{x,m}^{\rm{cl}}$ of $C_{x,m}$ is a projective curve with positive genus.  

Let $U\subset \cM_{H,\infty}^{\circ}$ be an affinoid neighborhood of $x$.  Then there exists $N\geq 0$ such that $Z_{x,m}\subset U$ for all $m\geq N$.  Let $K\subset \SL_2(\Q_p)$ be a compact open subgroup which stabilizes $U$, so that $U/K$ is an affinoid subset of the rigid-analytic curve $\cM_{H,\infty}^{\circ}/K$.    For each $m\geq N$, let $K_m\subset K$ be the stabilizer of $Z_{x,m}$, so that $K_m$ acts on $C_{x,m}$.  

There exists an integral model of $U/K$ whose special fiber contains as a component the completion of each $\overline{Z}_{x,m}/K_m$ which has positive genus.  Since there is a nonconstant morphism $\overline{Z}_{x,m}/K_m\to C_{x,m}/K_m$, we must have
\[ \dim_{\bF_\ell} H^1(U/K,\bF_\ell) \geq \sum_{m\geq N} \dim_{\bF_\ell} H^1(C_{x,m}^{\rm{cl}}/K_m,\bF_\ell). \]
Now we take a limit as $K$ shrinks.  Since $U\sim \varprojlim U/K$, we have $H^1(U,\bF_\ell)\isom \varinjlim H^1(U/K,\bF_\ell)$.  Also, for each $m$, the action of $K_m$ on $C_{x,m}$ is trivial for all sufficiently small $K$.  Therefore 
\[ \dim_{\bF_\ell} H^1(U,\bF_\ell) \geq \sum_{m\geq N} \dim_{\bF_\ell} H^1(C_{x,m}^{\rm{cl}},\bF_\ell) = \infty. \]
This shows that $U$ is not cohomologically smooth.

\medskip

\noindent \textbf{Acknowledgements.} The authors want to thank Peter Scholze for his help and his interest in their work. Also they thank Andreas Mihatsch for pointing out a mistake in a previous version of the manuscript. The first named author was supported by Peter Scholze's Leibniz Preis.   The second author was supported by NSF Grant No. DMS-1440140 while in residence at the Mathematical Sciences Research Institute in Berkeley, California.  

\bibliographystyle{amsalpha}
\bibliography{bibfile}

\providecommand{\bysame}{\leavevmode\hbox to3em{\hrulefill}\thinspace}
\providecommand{\MR}{\relax\ifhmode\unskip\space\fi MR }
\providecommand{\MRhref}[2]{%
  \href{http://www.ams.org/mathscinet-getitem?mr=#1}{#2}
}
\providecommand{\href}[2]{#2}
\begin{thebibliography}{{Sta}14}

\bibitem[BLR90]{NeronModels}
Siegfried Bosch, Werner L\"{u}tkebohmert, and Michel Raynaud, \emph{N\'{e}ron
  models}, Ergebnisse der Mathematik und ihrer Grenzgebiete (3) [Results in
  Mathematics and Related Areas (3)], vol.~21, Springer-Verlag, Berlin, 1990.

\bibitem[Che14]{ChenConnectedComponents}
Miaofen Chen, \emph{Composantes connexes g\'{e}om\'{e}triques de la tour des
  espaces de modules de groupes {$p$}-divisibles}, Ann. Sci. \'{E}c. Norm.
  Sup\'{e}r. (4) \textbf{47} (2014), no.~4, 723--764.

\bibitem[FF]{FarguesFontaineCurve}
Laurent Fargues and Jean-Marc Fontaine, \emph{Courbes et fibr\'es vectoriels en
  theorie de {H}odge $p$-adique}, To appear in {A}st\'erisque.

\bibitem[FS]{FarguesScholze}
Laurent Fargues and Peter Scholze, \emph{Geometrization of the local
  {L}anglands correspondence}, in preparation.

\bibitem[Kot85]{KottwitzIsocrystals}
Robert~E. Kottwitz, \emph{Isocrystals with additional structure}, Compositio
  Math. \textbf{56} (1985), no.~2, 201--220.

\bibitem[LB18]{LeBras}
Arthur-C\'{e}sar Le~Bras, \emph{Espaces de {B}anach--{C}olmez et faisceaux
  coh\'{e}rents sur la courbe de {F}argues--{F}ontaine}, Duke Math. J.
  \textbf{167} (2018), no.~18, 3455--3532.

\bibitem[Mes72]{Messing}
William Messing, \emph{The crystals associated to {B}arsotti-{T}ate groups:
  with applications to abelian schemes}, Lecture Notes in Mathematics, Vol.
  264, Springer-Verlag, Berlin-New York, 1972.

\bibitem[RZ96]{RapoportZink}
Michael Rapoport and Thomas Zink, \emph{Period spaces for {$p$}-divisible
  groups}, Annals of Mathematics Studies, vol. 141, Princeton University Press,
  Princeton, NJ, 1996.

\bibitem[Sch15]{ScholzeTorsion}
Peter Scholze, \emph{On torsion in the cohomology of locally symmetric
  varieties}, Ann. of Math. (2) \textbf{182} (2015), no.~3, 945--1066.

\bibitem[Sch17]{ScholzeEtaleCohomology}
\bysame, \emph{The \'etale cohomology of diamonds}, ARGOS Seminar in Bonn,
  2017.

\bibitem[{Sta}14]{StacksProject}
The {Stacks Project Authors}, \emph{\itshape {S}tacks {P}roject},
  \url{http://stacks.math.columbia.edu}, 2014.

\bibitem[SW13]{ScholzeWeinstein}
Peter Scholze and Jared Weinstein, \emph{Moduli of {$p$}-divisible groups},
  Camb. J. Math. \textbf{1} (2013), no.~2, 145--237.

\bibitem[Wei16]{WeinsteinSemistable}
Jared Weinstein, \emph{Semistable models for modular curves of arbitrary
  level}, Invent. Math. \textbf{205} (2016), no.~2, 459--526.

\end{thebibliography}

\end{document}